\documentclass[10pt,a4paper,leqno]{amsart}
\usepackage{a4wide}
\usepackage{amsmath,amsfonts,amssymb,latexsym}
\usepackage{hyperref}
\usepackage{xcolor}
\newcommand{\RR}{\mathbb{R}}
\newcommand{\CC}{\mathbb{C}}
\newcommand{\QQ}{\mathbb{Q}}
\newcommand{\NN}{\mathbb{N}}
\newcommand{\ZZ}{\mathbb{Z}}
\newcommand{\EE}{\mathbb{E}}
\newcommand{\Bo}{\mathcal{B}}
\newcommand{\Oo}{\mathcal{O}}

\newcommand{\RE}{ {\rm Re \,} }

\newtheorem{Th}{Theorem}
\newtheorem{Lem}{Lemma}
\newtheorem{Cor}{Corollary}
\newtheorem{Prop}{Proposition}

\theoremstyle{definition}
\newtheorem{Df}{Definition}

\theoremstyle{remark}
\newtheorem{Rem}{Remark}
\newtheorem{Ex}{Example}

\begin{document}


\title[Summable solutions of the Goursat problem for some PDEs with constant coefficients]{Summable solutions of the Goursat problem for some partial differential equations with constant coefficients}

\author{S{\l}awomir Michalik}

\address{Faculty of Mathematics and Natural Sciences,
College of Science\\
Cardinal Stefan Wyszy\'nski University\\
W\'oycickiego 1/3,
01-938 Warszawa, Poland\\
ORCiD: 0000-0003-4045-9548}
\email{s.michalik@uksw.edu.pl}
\urladdr{\url{http://www.impan.pl/~slawek}}

\begin{abstract}
We consider the Goursat problem for linear partial differential equations with constant coefficients in two complex variables.
We find the conditions for summable solutions of the Goursat problem in the case when the Newton polygon has exactly one side with a positive slope.
\end{abstract}

\maketitle

\section{Introduction}
  We study the Goursat problem for linear partial differential equations in two complex variables
\begin{equation}
  \label{eq:goursat_0}
  \left\{
   \begin{array}{l}
    P(\partial_{t},\partial_{z})u(t,z)=f(t,z)\\
     u(t,z)-v(t,z)=O(t^j z^{\alpha}).
   \end{array}
  \right.
 \end{equation}
 
 The general result on the unique solvability of the Goursat problem (\ref{eq:goursat_0}) in given functional space holds under the spectral condition (see \cite{G,H,W}), so in many papers this condition is posed as a~fundamental assumption. Such type results for functions  in Gevrey classes $\Oo[[t]]_s$ one can find for example in
 \cite{Miy2,Miy3,M-H}.

 On the other hand, after the first fundamental example given by Leray \cite{Le}, some papers 
as \cite{Al,Yo1,Yo2,Yo3} were devoted to study the existence and the uniqueness of the solution and the Fredholm property for the problem
(\ref{eq:goursat_0}) without assuming the spectral condition.

In author's opinion the most significant result in this direction was given by Miyake and Yoshino \cite{M-Y}. They employed the Toeplitz operator method to characterise Fredholm property and the unique solvability of the Goursat problem (\ref{eq:goursat_0}) in Gevrey class $\Oo[[t]]_s$ for a general linear operator $P(\partial_t,\partial_z)$ with holomorphic coefficients in a neighbourhood of the origin. It gives a far generalisation of Leray's result \cite{Le}.

In the paper we apply the result of Miyake and Yoshino \cite{M-Y} to study the conditions on the inhomogeneity $f(t,z)$ and the Goursat data $v(t,z)$ under which the solution of (\ref{eq:goursat_0}) is summable. Here we assume that $P(\partial_t,\partial_z)$ is an operator with constant coefficients and its Newton polygon has exactly one side with a positive slope. 

Borel summability of formal solutions of partial differential equations and its generalisations is a problem of significant interest in mathematics and there is a vast literature devoted to this question, see for example \cite{I2,La-M,La-Mi-Su,M-I,Re4,T-Y,Yo4} and references therein.

In the case of the operators with constant coefficients, the problem of characterisation of summable solutions of the Cauchy problem in terms of initial data was completely solved by Balser \cite{B5}
and the author \cite{Mic5}, but up to now there are no results about summable solutions of the Goursat problem. Only Remy \cite{Re3}
considered the Goursat problem in the context of Gevrey index and summability, but his result on summability holds only for solutions of the Cauchy problem.
For this reason one can treat our paper as the first step towards studying summable solutions of the Goursat problem for partial differential equations with constant coefficients.

 In the paper we use previously developed in \cite{Mic5} technique of factorisation of the operator $P(\partial_t,\partial_z)$ by simple
pseudodifferential operators $(\partial_t - \lambda(\partial_z))^l$,
and the deformation of contours in integral representations of solutions of simple pseudodifferential equations. 
The main steps in the proof of our result (Lemmas \ref{le:crucial} and \ref{le:crucial_3}) are based on precise estimations of exponential growth at infinity for analytic continuation of solutions of these equations. To this end we need to control not only the order of exponential growth but also its type.
We also use introduced by Balser and Yoshino \cite{B-Y} moment differential operators $\partial_{\Gamma_q,t}$ for sequence of moments $(m(n))_{n\in\NN}=(\Gamma(1+qn))_{n\in\NN}$, which arise after application of the Borel transform to the solution of the Goursat problem in Gevrey class $\Oo[[t]]_s$ with $q=1+s$.

The paper is organized as follows. First we collect notation and introduce the notion of Gevrey order, summability and moment differentiation $\partial_{\Gamma_q,t}$ with a sequence of moments $(m(n))_{n\in\NN}=(\Gamma(1+qn))_{n\in\NN}$. We also describe algebraic functions, study their behaviour at infinity and use them to construction of moment pseudodifferential operators. Next we discuss the initial and boundary conditions for the Goursat problem and we introduce the result of Miyake and Yoshino \cite{M-I} on the Fredholm property of the Goursat problem in Gevrey spaces. We also formulate their result in terms of the factorisation of the operator $P(\partial_t,\partial_z)$. 
Next we collect fundamental lemmas about the analytic continuation of the solutions of some equations. In the subsequent sections we apply these lemmas to prove the main results of the paper: Theorem \ref{th:3} about the characterisation of analytically continued solutions of the Goursat problem, and Theorem \ref{th:4}, which gives the conditions on summable solutions of the Goursat problem in terms of the analytic continuation property of the Borel transform of the inhomogeneity and the Goursat data.
\bigskip\par
\section{Notation}
$\mathbb{N}$ stands for the set of natural numbers $\{1,2,...\}$ and $\mathbb{N}_0:=\mathbb{N}\cup\{0\}$. We write $\mathbb{R}_+:=(0,+\infty)$ and $\mathbb{C}^*:=\mathbb{C}\setminus\{0\}$.

For any $r>0$, $D^n_r$ (resp. $\overline{D}^n_r$) stands for the open (resp. closed) disc in the complex space $\{z\in\CC^n\colon|z|<r\}$ (resp. $\{z\in\CC:|z|\le r\}$).
To simplify notation, we write $D_r$ (resp. $\overline{D}_r$) instead of $D^1_r$ (resp. $\overline{D}^1_r$).
 If the radius $r$ of the disc $D^n_r$ (resp. $\overline{D}^n_r$, $D_r$, $\overline{D}_r$) is unspecified,
we denote it briefly by $D^n$ (resp. $\overline{D}^n$, $D$, $\overline{D}$).

We write $\mathcal{R}$ for the Riemann surface of the logarithm. Let $\theta>0$ and $d\in\RR$. $S_{d}(\theta)$ denotes the open infinite sector in $\mathcal{R}$ with an opening $\theta$ and in a direction $d$ 
$$S_{d}(\theta):=\left\{z\in\mathcal{R}: |\hbox{arg}(z)-d|<\frac{\theta}{2}\right\}.$$
If the opening of the sector is unspecified we write $S_d$. 

We also write $\hat{S}_{d}(\theta;r):=S_{d}(\theta)\cup D_r$, and $\hat{S}_{d}(\theta)$ (resp. $\hat{S}_{d}$) if $r>0$ (resp. $r>0$ and $\theta>0$) are not specified. 

Let $S$ be the sector $S_d(\theta)$ (resp. the sum of sectors $\bigcup_{d,\theta}S_d(\theta)$).
The symbol $S'\prec S$ describes an infinite sector $S'$ (resp. a sum $S'$ of infinite sectors) with the vertex at the origin satisfying $\overline{S'}\subseteq S$.
Moreover, if additionally $\hat{S}=S\cup D_r$ (i.e. $\hat{S}=\hat{S}_{d}(\theta;r)$ or $\hat{S}=\bigcup_{d,\theta}\hat{S}_d(\theta;r)$) then
$\hat{S'}\prec \hat{S}$ means that $\hat{S'}=S'\cup D_{r'}$, with $S'\prec S$ and $0<r'<r$.

The set $\mathcal{O}(G)$ stands for the set of holomorphic functions on a domain $G\subseteq\CC^n$.
Analogously,
the set of holomorphic functions of the variables $z_1^{1/\kappa_{1}},\dots,z_n^{1/\kappa_n}$
($(\kappa_1,\dots,\kappa_n)\in\NN^n$) on $G$ is denoted by $\mathcal{O}_{1/\kappa_{1},\dots,1/\kappa_n}(G)$.
More generally, if $\EE$ denotes a Banach space with a norm $\|\cdot\|_{\EE}$, then
by $\Oo(G,\EE)$ (resp. $\Oo_{1/\kappa_{1},\dots,1/\kappa_n}(G,\EE)$) we shall denote the set of all $\EE$-valued  
holomorphic functions (resp. holomorphic functions of the variables $z_1^{1/\kappa_{1}},\dots,z_n^{1/\kappa_n}$) 
on a domain $G\subseteq\CC^n$.
For more information about functions with values in Banach spaces we refer the reader to \cite[Appendix B]{B2}. 
In the paper, as a Banach space $\EE$ we will
take the space of complex numbers $\CC$ (we abbreviate $\Oo(G,\CC)$ to $\Oo(G)$ and
$\Oo_{1/\kappa_{1},\dots,1/\kappa_n}(G,\CC)$ to $\Oo_{1/\kappa_{1},\dots,1/\kappa_n}(G)$)
or the space $\Oo$ of continuous functions on $\overline{D}$ and holomorphic on $D_r$ with a maximum norm.
\par
\begin{Df}
Let $S$ be a sector or a sum of sectors and $\hat{S}=S\cup D$. 
A function $u\in\Oo_{1/\kappa}(\hat{S},\EE)$
is of \emph{exponential growth of order at most $K\in\RR$ as $x\to\infty$
in $\hat{S}$} if for any
$\hat{S'}\prec\hat{S}$
there exist 
$A,B<\infty$ such that
\begin{gather*}
\|u(x)\|_{\EE}<Ae^{B|x|^K} \quad \textrm{for every} \quad x\in \hat{S'}.
\end{gather*}
The space of such functions is denoted by $\Oo_{1/\kappa}^K(\hat{S},\EE)$.
\par
Analogously, if $S_1$ and $S_2$ are sectors or sums of sectors then a function $u\in\Oo_{1/\kappa_1,1/\kappa_2}(\hat{S_1}\times\hat{S_2})$
is of \emph{exponential growth of order at most $(K_1,K_2)\in\RR^2$ as $(t,z)\to\infty$
in $\hat{S_1}\times\hat{S_2}$} if
for any
$\hat{S_i'}\prec\hat{S_i}$
($i=1,2$)
there exist $A,B_1,B_2<\infty$ such that
\begin{gather*}
|u(t,z)|<Ae^{B_1|t|^{K_1}}e^{B_2|z|^{K_2}} \quad \textrm{for every} \quad (t,z)\in \hat{S_1'}\times\hat{S_2'}.
\end{gather*}
The space of such functions is denoted by
$\Oo^{K_1,K_2}_{1/\kappa_1,1/\kappa_2}(\hat{S_1}\times\hat{S_2})$.
\end{Df}
\par
The space of formal power series
$ \widehat{u}(x)=\sum_{j=0}^{\infty}u_j x^{j}$ with $u_j\in\mathbb{E}$ is denoted by $\mathbb{E}[[x]]$.
Analogously, the space of formal power series $\widehat{u}(t,z)=\sum_{j,n=0}^{\infty}u_{jn} t^{j}z^{n}$ with
$u_{jn}\in\mathbb{E}$ is denoted by $\mathbb{E}[[t,z]]$.
\bigskip\par
\section{Formal power series: Gevrey order, summability and moment differentiation}
In this section we introduce some definitions and fundamental facts connected with $k$-summability and $\Gamma_q$-moment differentiation. For more details about summability we refer the reader to \cite{B2,LR}.
\begin{Df}
 Let $s\in\RR$. We say that $\widehat{u}(x)=\sum_{n=0}^{\infty}u_nx^n\in\EE[[x]]$ is a \emph{formal power series of Gevrey order $s$} if there exists $A,B<\infty$ such that
 \begin{gather*}
  \|u_n\|_{\EE}\leq A B^n n!^s\quad\textrm{for every}\quad n\in\NN_0.
 \end{gather*}
The space of such series is denoted by $\EE[[x]]_s$.
\end{Df}

\begin{Rem}
 If $\widehat{u}(x)\in\EE[[x]]_s$ and $s\leq 0$ then $\widehat{u}(x)$ is convergent and its sum $u(x)$ is well defined. Moreover $u(x)\in\Oo(D,\EE)$ for $s=0$ and $u(x)\in\Oo^{-1/s}(\CC,\EE)$ for $s<0$. 
\end{Rem}

\begin{Df}
 Let $s\geq 0$ and $\Gamma(\cdot)$ be the gamma function. The linear operator $\Bo_s\colon\EE[[t]]\to\EE[[t]]$
 defined by
 $$
 \Bo_s\Big(\sum_{n=0}^{\infty}\frac{u_n}{n!}t^n\Big):=\sum_{n=0}^{\infty}\frac{u_n}{\Gamma(1+(s+1)n)}t^n
 $$
 is called a \emph{modified Borel transform of order $s$}.
\end{Df}

\begin{Rem}
\label{re:convergent}
 Observe that $\widehat{u}(t)\in\EE[[t]]_s$ if and only if its modified Borel transform $\Bo_s\widehat{u}(t)$ is convergent in some complex neighbourhood of the origin.
\end{Rem}

\begin{Df}
 Let $s>0$, $k=1/s$, $d\in\RR$. We say that a formal power series 
 $\widehat{u}(t)\in\EE[[t]]$ is \emph{$k$-summable in a direction $d$} if there exists $\hat{S}_d$ such that $\Bo_s\widehat{u}(t)$
 is convergent and its sum belongs to the space $\Oo^{k}(\hat{S}_d,\EE)$.
\end{Df}

\begin{Rem}
 If $\widehat{u}(t)$ is $k$-summable in some direction $d$ then
 by Remark \ref{re:convergent} $\widehat{u}(t)\in\EE[[t]]_s$.
\end{Rem}

\begin{Df}
 Let $q>0$. The linear operator $\partial_{\Gamma_q,t}\colon \EE[[t]]\to\EE[[t]]$ defined by
 $$
 \partial_{\Gamma_q,t}\Big(\sum_{n=0}^{\infty}\frac{u_n}{\Gamma(1+qn)}t^n\Big):=\sum_{n=0}^{\infty}\frac{u_{n+1}}{\Gamma(1+qn)}t^n
 $$
 is called a \emph{$\Gamma_q$-moment differentiation}.
\end{Df}
\begin{Rem}
 Observe that $\Gamma_1$-moment differentiation $\partial_{\Gamma_1,t}$ coincides with the standard differentiation $\partial_t$ on the space of formal power series $\EE[[t]]$.
\end{Rem}
\begin{Rem}
 \label{re:5}
 Direct calculation shows that modified Borel transform and moment differentiation satisfy the following commutation type formula
 \begin{gather*}
 \Bo_s\partial_t \widehat{u}(t)=\partial_{\Gamma_q,t}\Bo_s\widehat{u}(t)
 \quad\textrm{for every}\quad
 \widehat{u}(t)\in\EE[[t]],
 \end{gather*}
  where $s\geq 0$ and $q=1+s$.
\end{Rem}
\bigskip\par
\section{Algebraic functions and moment pseudodifferential operators}
In this section we collect some facts about algebraic functions and next we introduce moment pseudodifferential operators for such functions.

Let $\lambda(\zeta)$ be an algebraic function on $\CC$. It means that there exists a polynomial $P(\lambda,\zeta)$ of two complex variables such that the function $\lambda(\zeta)$ satisfies equation $P(\lambda(\zeta),\zeta)=0$. By the implicit function theorem the function $\lambda(\zeta)$ is holomorphic on $\CC$ but a finite number of singular or branching points. Moreover this function has a moderate growth at infinity. More precisely there exist a \emph{pole order at infinity}
$q\in\QQ$ and a \emph{leading term} $\lambda\in\CC^*$
such that 
$$\lim_{\zeta\to\infty}\frac{\lambda(\zeta)}{\zeta^q}=\lambda.$$
We denote it shortly by $\lambda(\zeta)\sim\lambda\zeta^q$.

Hence there exists $r_0<\infty$ and $\kappa\in\NN$ such that $\lambda(\zeta)$ is a holomorphic function of the variable $\xi=\zeta^{1/\kappa}$ for $|\zeta|>r_0$ with a pole at infinity. It means that the function
$\xi\mapsto \lambda(\xi^{\kappa})$ has the Laurent series expansion $\lambda(\xi^{\kappa})=\sum_{j=-n}^{\infty}\frac{a_j}{\xi^j}$ at infinity for some coefficients $a_j\in\CC$ with $a_{-n}=\lambda$ and $n=q\kappa\in\ZZ$. This expansion is convergent for $|\xi|>r_0^{1/\kappa}$ with a pole of order $n$ at infinity.

We will show 
\begin{Lem}
\label{le:inverse}
 If $\lambda(\zeta)$ is an algebraic function such that $\lambda(\zeta)\sim \lambda\zeta$ for some $\lambda\in\CC^*$  then there exists $\kappa\in\NN$ and $r_0>0$ such that the function $\tilde{\lambda}(\zeta):=\Big(\lambda(\zeta^{\kappa})\Big)^{1/\kappa}$ is holomorphic and invertible for $|\zeta|>r_0$ with a simple pole at infinity.
\end{Lem}
\begin{proof}
Since $\lambda(\zeta)$ is an algebraic function, it is a holomorphic function with a finite number of singular points in $\CC\cup\{\infty\}$, in which the function has algebraic poles or algebraic branching points. It means that there exists $\tilde{r}_0>0$ and $\kappa\in\NN$ such that the function $\zeta\mapsto\lambda(\zeta^{\kappa})$ is holomorphic for $|\zeta|>\tilde{r}_0$ with a pole of order $\kappa$ at infinity. Hence, by the Laurent series expansion
$$
\lambda(\zeta^{\kappa})=\sum_{n=-\kappa}^{\infty}\frac{a_n}{\zeta^n}=
\zeta^{\kappa}\sum_{n=0}^{\infty}\frac{a_{n-\kappa}}{\zeta^n}=:\zeta^{\kappa}g(\zeta)\quad\textrm{with}\quad a_{-\kappa}=\lambda\in\CC^*,
$$
where $g(\zeta)$ is holomorphic in a complex neighbourhood of infinity (for $|\zeta|>\tilde{r}_0$) such that $\lim\limits_{z\to\infty}g(z)= \lambda\in\CC^*$.
It means that $g(\zeta)$ has a holomorphic branch of $\kappa$-th root of $g(\zeta)$ at a neighbourhood of infinity.

Hence the function $\tilde{\lambda}(\zeta)$ is holomorphic with a simple pole at infinity and its Laurent expansion is given by
$$
\tilde{\lambda}(\zeta)=\zeta g(\zeta)^{1/\kappa}=\lambda^{1/\kappa}\zeta
+\sum_{n=0}^{\infty}\frac{b_n}{\zeta^n}.
$$
Since additionally $\lim\limits_{\zeta\to\infty}\frac{d}{d\zeta}\tilde{\lambda}(\zeta)=\lambda^{1/\kappa}\neq 0$, by the inverse function theorem we conclude that the function $\tilde{\lambda}(\zeta)$ is invertible for
$|\zeta|>r_0$ (for some sufficiently large $r_0>\tilde{r}_0$) and its inversion
$\tilde{\lambda}^{-1}(\zeta)$ is also holomorphic in a complex neighbourhood of infinity (say, for some $r_1>0$ and for every 
$|\zeta|>r_1$) and $\tilde{\lambda}^{-1}(\zeta)\sim\lambda^{-1/\kappa}\zeta$.
\end{proof}

\begin{Rem}
 The above lemma allows us to replace an algebraic function $\lambda(\zeta)$ with an algebraic branching point of order $\kappa$ at infinity by
 the algebraic function $\tilde{\lambda}(\zeta)$ with a simple pole at infinity. It will be used in Lemma \ref{le:crucial_2}.
\end{Rem}
\bigskip\par
To construct $\Gamma_{1/p}$-moment pseudodifferential operators,
first we recall the integral representation of 
$\partial^n_{\Gamma_{1/p},z}\varphi(z)$
\begin{Prop}[{\cite[Proposition 3]{Mic7}}]
Let $\varphi(z)\in\Oo(D_r)$ and $p\in\NN$. Then for every $|z|<\varepsilon<r$ and $n\in\NN$ we have
\begin{gather}
\label{eq:integral_representation}
 \partial^n_{\Gamma_{1/p},z}\varphi(z)=\frac{1}{2\pi i}\oint_{|w|=\varepsilon}\varphi(w)\int_{0}^{e^{i\theta}\infty}\zeta^n\mathbf{E}_{1/p}(z\zeta)p(w\zeta)^{p-1}e^{-(\zeta w)^{p}}\,d\zeta dw,
\end{gather}
where $\theta\in(-\arg w -\frac{\pi}{2p}, -\arg w +\frac{\pi}{2p})$ and $\mathbf{E}_{1/p}(z):=\sum_{n=0}^{\infty}\frac{z^n}{\Gamma(1+n/p)}$ is the Mittag-Leffler function of index $1/\kappa$.
\end{Prop}

Using the formula (\ref{eq:integral_representation}) for any algebraic function $\lambda(\zeta)$, which is also holomorphic with respect to $\zeta$ for $|\zeta|\geq r_0$, one can define a moment pseudodifferential operator $\lambda(\partial_{\Gamma_{1/p},z})\colon\Oo(D)\to\Oo(D)$ as (see \cite[Definition 8]{Mic7})
\begin{gather*}
 \lambda(\partial_{\Gamma_{1/p},z})\varphi(z):=\frac{1}{2\pi i}\oint_{|w|=\varepsilon}\varphi(w)\int_{e^{i\theta}r_0}^{e^{i\theta}\infty}\lambda(\zeta)\mathbf{E}_{1/p}(z\zeta)p(w\zeta)^{p-1}e^{-(\zeta w)^{p}}\,d\zeta dw
\end{gather*}
for $\theta\in(-\arg w -\frac{\pi}{2p}, -\arg w +\frac{\pi}{2p})$.

We extend this definition to the case when $\lambda(\zeta)$ is a general algebraic function, i.e. holomorphic of the variable $\xi=\zeta^{1/\kappa}$ for $|\zeta|\geq r_0$ (for some $\kappa\in\NN$ and $r_0>0$). Since by \cite[Lemma 3]{Mic7} we see that
$(\partial_{\Gamma_{1/p},z}\varphi)(z^{\kappa})=\partial^{\kappa}_{\Gamma_{1/p\kappa},z}(\varphi(z^{\kappa}))$, the operator $\lambda(\partial_{\Gamma_{1/p},z})$ should satisfy the formula
\begin{gather*}
 (\lambda(\partial_{\Gamma_{1/p},z})\varphi)(z^{\kappa})=
 \lambda(\partial^{\kappa}_{\Gamma_{1/p\kappa},z})(\varphi(z^{\kappa}))\quad\textrm{for every}\quad\varphi(z)\in\Oo_{1/\kappa}(D).
\end{gather*}
For this reason we define

\begin{Df}[see {\cite[Definition 13]{Mic8}}]
Let $p\in\NN$ and $\lambda(\zeta)$ be a holomorphic function of the variable $\xi=\zeta^{1/\kappa}$ for $|\zeta|\geq r_0$ (for some $\kappa\in\NN$ and $r_0>0$) and of moderate growth at infinity. A \emph{moment pseudodifferential operator} $\lambda(\partial_{\Gamma_{1/p},z})\colon \Oo_{1/\kappa}(D)\to\Oo_{1/\kappa}(D)$ is defined by
\begin{gather}
\label{eq:lambda}
\lambda(\partial_{\Gamma_{1/p},z})\varphi(z):=\frac{1}{2\kappa\pi i}\oint^{\kappa}_{|w|=\varepsilon}\varphi(w)\int_{e^{i\theta}r_0}^{e^{i\theta}\infty}\lambda(\zeta)\mathbf{E}_{1/p\kappa}(z^{1/\kappa}\zeta^{1/\kappa})p(w\zeta)^{p-1}e^{-(\zeta w)^{p}}\,d\zeta dw
\end{gather}
for every $\varphi(z)\in\Oo_{1/\kappa}(D_r)$ and $|z|<\varepsilon<r$, where $\theta\in(-\arg w -\frac{\pi}{2p}, -\arg w +\frac{\pi}{2p})$ and $\oint^{\kappa}_{|w|=\varepsilon}$ means that we integrate $\kappa$ times along the positively oriented circle of radius $\varepsilon$.
\end{Df}
We show
\begin{Prop}
 \label{pr:2}
 The right-hand side of (\ref{eq:lambda}) does not depend on the choice of the number $r_0$ such that $\lambda(\zeta)$ is holomorphic for $|\zeta|\geq r_0$.
\end{Prop}
\begin{proof}
 We take $r_1\leq r_2$ such that $\lambda(\zeta)$ is holomorphic
 of the variable $\xi=\zeta^{1/\kappa}$ for $|\zeta|\geq r_1$.
 
Observe that for every fixed
 $z\in D_{\varepsilon}$ the function
 \begin{gather}
 \label{eq:function_w}
  w\longmapsto \int_{r_1e^{i\theta}}^{r_2e^{i\theta}}\lambda(\zeta)\mathbf{E}_{1/p\kappa}(z^{1/\kappa}\zeta^{1/\kappa})p(w\zeta)^{p-1}e^{-(\zeta w)^{p}}\,d\zeta
 \end{gather}
 is holomorphic for $w\in D_r\setminus\{0,z\}$ and $\theta\in(-\arg w -\frac{\pi}{2p}, -\arg w +\frac{\pi}{2p})$.
 Moreover, the function (\ref{eq:function_w}) extends continuously to the whole disc $D_r$. Hence, by the Cauchy integral formula
 \begin{gather*}
  \oint^{\kappa}_{|w|=\varepsilon}\varphi(w)\int_{e^{i\theta}r_1}^{e^{i\theta}r_2}\lambda(\zeta)\mathbf{E}_{1/p\kappa}(z^{1/\kappa}\zeta^{1/\kappa})p(w\zeta)^{p-1}e^{-(\zeta w)^{p}}\,d\zeta dw=0,
 \end{gather*}
so the right-hand side of (\ref{eq:lambda}) is independent of the choice of $r_0$.
\end{proof}
\begin{Rem}
 By Proposition \ref{pr:2} the value of $\lambda(\partial_{\Gamma_{1/p},z})\varphi(z)$ depends only on $\varphi(z)$ and on the behaviour of the algebraic function $\lambda(\zeta)$ at a neighbourhood of infinity.
\end{Rem}
\bigskip\par
\section{Conditions in the Goursat problem}
In this section we discuss the initial and boundary conditions in the Goursat problem
\begin{equation}
\label{eq:goursat_1}
  \left\{
   \begin{array}{l}
    P(\partial_{t},\partial_{z})u(t,z)=f(t,z)\in\CC[[t,z]]\\
    \partial_t^k u(0,z)=\varphi_k(z)\in\CC[[z]],\ \ k=0,...,j-1\\
    \partial_z^{\beta} u(t,0)=\psi_{\beta}(t)\in\CC[[t]],\ \ \beta=0,...,\alpha-1.
   \end{array}
  \right.
 \end{equation}
The initial and boundary data have to satisfy the following compatibility conditions for $k=0,...,j-1$ and $\beta=0,...,\alpha-1$:
\begin{equation}
\label{eq:conditions}
\varphi_k^{(\beta)}(0)=\psi_{\beta}^{(k)}(0)=\partial_t^k\partial_z^{\beta}u(0,0)=:c_{k\beta}.
\end{equation}
Observe that for any initial and boundary data satisfying (\ref{eq:conditions}) we may always find the formal power series
$v(t,z)$ with the same initial and boundary conditions as $u(t,z)$ in
(\ref{eq:goursat_1})
\begin{equation}
 \label{eq:cond_v}
  \left\{
   \begin{array}{l}
    \partial_t^k v(0,z)=\varphi_k(z),\ \ k=0,...,j-1\\
    \partial_z^{\beta} v(t,0)=\psi_{\beta}(t),\ \ \beta=0,...,\alpha-1.
   \end{array}
  \right.
 \end{equation}
Precisely, $v(t,z)\in\CC[[t,z]]$ satisfies the conditions (\ref{eq:cond_v}) if and
only if the formal power series $v(t,z)$ is given by
\begin{equation}
\label{eq:any_v}
v(t,z)=\sum_{k=0}^{j-1}\frac{\varphi_k(z)}{k!}t^k+
\sum_{\beta=0}^{\alpha-1}\frac{\psi_{\beta}(t)}{\beta!}z^{\beta}-\sum_{k=0}^{j-1}\sum_{\beta=0}^{\alpha-1}\frac{c_{k\beta}}{k!\beta!}t^kz^{\beta}+t^k z^{\alpha} r(t,z)
\end{equation}
for some $r(t,z)\in\CC[[t,z]]$.
\begin{Df}
 A formal power series $v(t,z)\in\CC[[t,z]]$ satisfying (\ref{eq:any_v}) with $r(t,z)=0$ will be called a \emph{Goursat data} for the problem (\ref{eq:goursat_1}).
\end{Df}
Using a Goursat data we can write the problem (\ref{eq:goursat_1}) as
\begin{equation*}
  \left\{
   \begin{array}{l}
    P(\partial_{t},\partial_{z})u(t,z)=f(t,z)\\
     u(t,z)-v(t,z)=O(t^j z^{\alpha}),
   \end{array}
  \right.
 \end{equation*}
where the condition $u(t,z)-v(t,z)=O(t^j z^{\alpha})$ means that 
$$
\frac{u(t,z)-v(t,z)}{t^j z^{\alpha}}\in\CC[[t,z]].
$$
\bigskip\par
\section{Fredholm property of the Goursat problem}
In this section we recall the important result of Miyake and Yoshino \cite{M-Y} about the Fredholm property of the
Goursat problem in Gevrey spaces. We present this result in the
special case of equations with constant coefficients.

Namely,
we consider the Goursat problem for general linear partial differential equations with constant coefficients in Gevrey spaces $\Oo[[t]]_s$ for fixed $s\geq 0$:
\begin{equation}
  \label{eq:goursat}
  \left\{
   \begin{array}{l}
    P(\partial_{t},\partial_{z})u(t,z)=f(t,z)\in\Oo[[t]]_s\\
     u(t,z)-v(t,z)=O(t^j z^{\alpha}),\ \ v(t,z)\in\Oo[[t]]_s,
   \end{array}
  \right.
 \end{equation}
where 
\begin{equation*}
 P(\partial_t,\partial_z):=\sum_{(j,\alpha)\in\Lambda}a_{j\alpha}\partial_t^j\partial_z^{\alpha}
\end{equation*}
and $\Lambda\subseteq\NN_0\times\NN_0$ is a finite set of indices.

To formulate the main result of Miyake and Yoshino \cite{M-Y}, first we define a Newton polygon $N(P)$ for the operator $P(\partial
_t,\partial_z)$ as the convex hull of the union of sets $Q(j+\alpha,-j)$ for $(j,\alpha)\in\Lambda$
$$N(P):={\rm conv\,}\{Q(j+\alpha,-j)\colon\ (j,\alpha)\in \Lambda,\ a_{j\alpha}\neq 0\},$$
where $Q(a,b):=\{(x,y)\in\RR^2\colon\ x\leq a,\ y\geq b\}$.

The Newton polygon is a classical tool introduced in \cite{Y} and representing properties of partial differential operators in a geometric way.

For a given $s\geq 0$, we draw a line $L_s$ with slope $k=1/s\in\RR_+\cup\{\infty\}$ such that 
$N_s:=N(P)\cap L_s$ is a nonempty set contained in
$\partial N(P)$.

Next, let 
$$\mathring{N}_s:=\{(j,\alpha)\in\Lambda\colon\ a_{j\alpha}\neq 0,\ (j+\alpha,-j)\in N_s\}.
$$

\begin{Df}
The \emph{principal part} $P_s(\partial_t,\partial_z)$ and the \emph{Toeplitz symbol} $f_s(z)$ associated with the Gevrey index $s$ are defined by
$$P_s(\partial_t,\partial_z):=\sum_{(j,\alpha)\in\mathring{N}_s}a_{j\alpha}\partial_t^j\partial_z^{\alpha}\quad\textrm{and}\quad
f_s(z):=\sum_{(j,\alpha)\in\mathring{N}_s}a_{j\alpha}z^{-j}=P_s(z^{-1},1).
$$
\end{Df}

\begin{Df}
Let $f(z)=\sum_{j=-m}^nf_jz^j\in\CC[z,z^{-1}]$. An infinite matrix $T_f$ defined by $T_f:=(f_{j-k})_{j,k\in\NN_0}$ is said to be a \emph{Toeplitz matrix with symbol $f(z)$}. Similarly, 
an \emph{$N$-finite section Toeplitz matrix $T_f(N)$ with symbol $f(z)$} is defined by $T_f(N):=(f_{j-k})_{j,k=0,1,\dots,N}$ for any $N\in\NN_0$. 
\end{Df}

For given $w,R>0$ we introduce the Banach space $G^s_w(R)$ associated with $\Oo[[t]]_s$ as follows
\begin{gather*}
G^s_w(R):=\big\{u(t,z)=\sum_{k,\beta=0}^{\infty}u_{k\beta}t^{k}z^{\beta}/k!\beta!\in\Oo[[t]]_s\colon\ 
\|u\|^{(s)}_{w,R}:=\sum_{k,\beta=0}^{\infty}|u_{k\beta}|w^{k}R^{sk+\beta}/(sk+\beta)!<\infty\big\}.
\end{gather*}

\begin{Rem}
Observe that for any $w>0$ and any $u(t,z)\in\Oo[[t]]_s$ there exists sufficiently small $R_0>0$ such that $u(t,z)\in G^s_w(R)$
for any $0<R\leq R_0$.
\end{Rem}

\begin{Df}
 We say that an operator $L(\partial_t,\partial_z)$ \emph{has Fredholm property on $G^s_w$} if
 there exists $R_0>0$ such that  for any
$0<R\leq R_0$ the operator
$L(\partial_t,\partial_z)\colon G^s_w(R)\to G^s_w(R)$
is a Fredholm operator on $G^s_w(R)$ with the index equal to zero (i.e. $L(\partial_t,\partial_z)$ has the same finite dimensional kernel and cokernel). 
\end{Df}

We have
\begin{Th}[{\cite[Theorem 0]{M-Y}}]
\label{th:1}
Suppose $w>0$, $\mathring{N}_s\neq \emptyset$, $(j,\alpha)\in{\rm conv\,}\{\mathring{N}_s\}$ and $f(z)=f_s(z)z^j$.
Then the condition
\begin{equation}
 \label{eq:H_w}
 \tag{$H_w$}
 f_s(z)\neq 0\ \textrm{on}\ |z|=w\
\textrm{and}\ I_w(f)=\oint_{|z|=w}d(\log f(z))=0.
\end{equation}
is satisfied if and only if for sufficiently small $R>0$ the operator
\begin{equation*}
L(\partial_t,\partial_z):=
P(\partial_t,\partial_z)\partial_t^{-j}\partial_z^{-\alpha}\colon G^s_w(R)\to G^s_w(R)
\end{equation*}
has Fredholm property on $G^s_w$.

Moreover this Fredholm operator $L(\partial_t,\partial_z)$ is bijective if and only if additionally $N$-th finite section Toeplitz matrix $T_{f}(N)$ with symbol $f(z)$ is invertible for any $N\in\NN_0$.

Hence, in particular, if one of the following conditions holds then $L(\partial_t,\partial_z)$ is a bijection:
\begin{enumerate}
 \item[(i)] $(j,\alpha)$ is an end point of the interval ${\rm conv\,}\{\mathring{N}_s\}$.
 \item[(ii)] there exists $c>0$ such that $0\not\in{\rm conv\,}\{f(z)\colon |z|=c\}$.
\end{enumerate} 
\end{Th}

\begin{Rem}
 The condition $I_w(f)=0$  for $f(z)=z^jf_s(z)$ in (\ref{eq:H_w}) means that the winding number of $f(z)$ at the origin with respect to the circle $K_{w}:=\{z\in\CC\colon |z|=w\}$ is equal to zero. By the argument principle this condition means that the function $f(z)$ has the same total number of zeros and poles (counted according to their multiplicity) inside the contour $K_w$. 
 \end{Rem}
 
 \begin{Rem}
  The condition (\ref{eq:H_w}) means that $N$-th finite section Toeplitz matrix $T_{f}(N)$ is invertible for sufficiently large $N\in\NN_0$,
  but this condition does not guarantee that it holds for any $N\in\NN_0$.
 \end{Rem}

\begin{Rem}
Observe that if the Fredholm operator $L(\partial_t,\partial_z)$ is bijective then the Goursat problem (\ref{eq:goursat}) is uniquely solvable in the Gevrey space $\Oo[[t]]_s$. So, the above theorem gives the sufficient conditions for the unique solvability of (\ref{eq:goursat}) in $\Oo[[t]]_s$.
\end{Rem}

Using Theorem \ref{th:1} it is easy to prove the spectral condition, which also guarantees 
the uniqueness and solvability of the Goursat problem (\ref{eq:goursat}).
\begin{Cor}[The spectral condition, see \cite{G,H,W}]
 If $(j,\alpha)\in\mathring{N}_s$ and the coefficients of the
 principal part $P_s(\partial_t,\partial_z)$ satisfy the spectral condition
 \begin{equation}
  \label{eq:spectral}
  |a_{j\alpha}|>\sum_{(l,\beta)\in\mathring{N}_s\setminus (j,\alpha)}|a_{l\beta}|w^{j-l}\quad\textrm{for some}\quad w>0,
 \end{equation}
 then the Goursat problem (\ref{eq:goursat}) is uniquely solvable in the Gevrey space $\Oo[[t]]_s$.
\end{Cor}
\begin{proof}
 If we expand $f(z):=z^jf_s(z)$ in the Laurent series $f(z)=\sum_{k=-m}^nf_kz^k$ then the spectral condition (\ref{eq:spectral}) means that
 $$|f_0|>\sum_{\genfrac{}{}{0pt}{}{k=-m}{k\neq0}}^n |f_k|w^k\quad\textrm{for some}\quad w>0.$$
 Using the Rouch\'e theorem we conclude that (\ref{eq:H_w})
 is satisfied for the same $w>0$.
 
 Moreover, we will show that $f(z)$ satisfies (ii).
 Indeed, multiplying $f(z)$ by $e^{i\theta}$ for some $\theta\in\RR$, if necessary, we may assume that $|f_0|=f_0$. Then
 $$
 \RE f(z)>f_0-\sum_{\genfrac{}{}{0pt}{}{k=-m}{k\neq0}}^n |f_k|w^k>0\quad \textrm{on the circle}\quad |z|=w.
 $$
 It means that $0\not\in\textrm{conv\,}\{f(z)\colon |z|=w\}$ and (ii) holds.
\end{proof}
\bigskip\par
\section{Fredholm property and the factorisation of the operator}
In this section we reformulate the result of Miyake and Yoshino \cite{M-Y} in terms of properties of the factorisation of
the operator $P(\partial_t,\partial_z)$.

Let $P(\lambda,\zeta)$ be a general polynomial of two variables, which is of order $m$ with respect to $\lambda$. We may write it as
 \begin{equation*}
     P(\lambda,\zeta)=P_0(\zeta)\lambda^m-\sum_{j=1}^m P_j(\zeta)\lambda^{m-j}=P_0(\zeta)\prod_{j=1}^n\prod_{k=1}^{m_j}(\lambda-\lambda_{jk}(\zeta)),
 \end{equation*}
 where $P_0(\zeta)\sim a_0\zeta^{m_0}$ for some $a_0\in\CC^*$ and $m_0\in\NN_0$, $m_1+\dots+m_n=m$ and $\lambda_{jk}(\zeta)$ are the roots of the characteristic equation
   $P(\lambda,\zeta)=0$ satisfying $\lambda_{jk}(\zeta)\sim \lambda_{jk}\zeta^{q_j}$ for some $\lambda_{jk}\in\CC^*$ and $q_j\in\QQ$.
   
Additionally we may assume that $q_1>q_2>...>q_n$ and 
$|\lambda_{j1}|\geq |\lambda_{j2}|\geq\dots\geq |\lambda_{jm_j}|$
for $j=1,\dots,n$.

Further, let 
$$
\tilde{n}:=
\left\{
   \begin{array}{lll}
 0&\textrm{for}& q_1<1\\
\max\{1\leq j\leq n\colon q_j\geq 1\}&\textrm{for}& q_1\geq 1.
\end{array}
\right.
$$

Then the Newton polygon of the operator $P$ has $\tilde{n}$ sides with
positive (or positive and vertical) slopes. These slopes are given by $1/s_1,\dots,
1/s_{\tilde{n}}$, where $s_l=q_l-1$ for $l=1,\dots,\tilde{n}$.

To shorten notation we denote
$\Lambda_l:=\prod_{j=1}^{l}\prod_{k=1}^{m_j}\lambda_{jk}$
for $l=1,...,\tilde{n}$ and $\Lambda_0:=1$.

We show that

\begin{Th}
\label{th:2}
   Suppose that $s\in[0,\infty)$, $(j,\alpha)\in{\sf conv\,}\{\mathring{N}_s\}$ and 
   $$L(\partial_t,\partial_z):=
    P(\partial_t,\partial_z)\partial_t^{-j}\partial_z^{-\alpha}
    \colon G^s_w(R)\to G^s_w(R).$$
    Then the following hold:
\begin{enumerate}
\item If $s=s_l$ for some $l\in\{1,\dots,\tilde{n}\}$ then
there exist $w>0$ such that for sufficiently small $R$ the operator $L(\partial_t,\partial_z)$ has Fredholm property on $G^s_w$ if and only if $|\lambda_{l,j_l}|>|\lambda_{l,j_{l}+1}|$ and $w\in(|\lambda_{l,j_{l}}|^{-1},|\lambda_{l,j_l+1}|^{-1})$.
Moreover $L(\partial_t,\partial_z)$ is bijective if and only if additionally $N$-th finite section Toeplitz matrix $T_{f}(N)$ with symbol $f(z):=f_s(z)z^j$ is invertible for any $N\in\NN_0$.
\item If $s\not\in \{s_1,\dots,s_{\tilde{n}}\}$ then for any $w>0$ there exists $R_0>0$ such that for any $R\in(0,R_0]$ the operator $L(\partial_t,\partial_z)$ is a bijection on $G^s_w(R)$.
\end{enumerate}
\end{Th}
\begin{proof}
 Observe that
$$
P_{s_l}(\partial_t,\partial_z)= a_0\Lambda_{l-1}\partial_t^{m-m_1-\dots-m_l}\partial_z^{m_0+q_1m_1+\dots+q_{l-1}m_{l-1}}\prod_{k=1}^{m_l}(\partial_t-\lambda_{lk}\partial_z^{q_l})$$
for $l=1,\dots,\tilde{n}$.

Moreover, if $s\in(s_{l+1},s_{l})$ for $l=0,\dots,\tilde{n}$ with
$s_0:=\infty$ and $s_{\tilde{n}+1}:=0$,
then 
$$P_{s}(\partial_t,\partial_z)= a_0\Lambda_l\partial_t^{m-m_1-\dots-m_l}\partial_z^{m_0+q_1m_1+\dots+q_{l}m_{l}}.$$
Additionally, if $\tilde{n}=0$ or $q_{\tilde{n}}>1$ then we have
$$P_0(\partial_t,\partial_z)=a_0\Lambda_{\tilde{n}}\partial_t^{m-m_1-\dots-m_{\tilde{n}}}\partial_z^{m_0+q_1m_1+\dots+q_{\tilde{n}}m_{\tilde{n}}}.$$
\bigskip
\par
Fix $s\geq 0$ and $(j,\alpha)\in{\sf conv\,}\{\mathring{N}_s\}$.
We have 3 possibilities:
\begin{enumerate}
\item If $s=s_l$ for some $l\in\{1,\dots,\tilde{n}\}$  then  $j=m-m_1-\cdots-m_{l-1}-j_l$ for some $j_l\in\{0,\dots,m_l\}$ and
$$
f_s(z)z^j=a_0\Lambda_{l-1}z^{-j_l}\prod_{k=1}^{m_l}(1-\lambda_{lk}z).$$
This function has zeros at points $\lambda_{l1}^{-1},\dots,\lambda_{lm_l}^{-1}$, so 
the condition (\ref{eq:H_w}) is satisfied for some $w>0$ if and only if $|\lambda_{l,j_l}|>|\lambda_{l,j_{l}+1}|$ and $w\in(|\lambda_{l,j_{l}}|^{-1},|\lambda_{l,j_{l}+1}|^{-1})$ with $\lambda_{l,0}:=\infty$ and $\lambda_{l,m_l+1}:=0$.
\item If  $s\in(s_{l+1},s_{l})$ for some $l\in\{0,1,\dots,\tilde{n}\}$
then $j=m-m_1-\cdots-m_l$ and
$f_s(z)z^j=a_0\Lambda_{l}$.
Since this function is constant and different than zero, we conclude that the condition (\ref{eq:H_w}) is satisfied for any $w>0$, and moreover
by (ii) the problem (\ref{eq:goursat}) is uniquely solved.
\item If $s=0$ and additionally $q_{\tilde{n}}>1$ or $\tilde{n}=0$ then
$j=m-m_1-\cdots-m_{\tilde{n}}$ and
$f_s(z)z^j=a_0\Lambda_{\tilde{n}}$,
so we have the same conclusion as in the previous case.
\end{enumerate}
\end{proof}
\bigskip\par
\section{Crucial lemmas}
In this section we collect the crucial lemmas about the analytic continuation of the solutions of some equations. 

We fix $a,b\in\QQ$ such that $a>b>0$ and $a\geq 1$. We also assume that the symbol of the operator $P(\partial_{\Gamma_a,t},\partial_z)$ has the factorisation
  $P(\lambda,\zeta)=P_{(1)}(\lambda,\zeta) P_{(2)}(\lambda,\zeta)$,
\begin{gather*}
P_{(i)}(\lambda,\zeta)=\prod_{k=1}^{m_i}(\lambda-\lambda_{ik}(\zeta))^{\alpha_{ik}},
\quad i=1,2,
\end{gather*}
with algebraic functions $\lambda_{ik}(\zeta)$ which are also analytic functions of the variable $\xi=\zeta^{1/\kappa}$ for $|\zeta|\geq r_0$ (for some $\kappa\in\NN$ and $r_0>0$), such that
$\lambda_{1k}(\zeta)\sim \lambda_{1k}\zeta^a$
and
$\lambda_{2k}(\zeta)\sim \lambda_{2k}\zeta^{q_k}$, $q_k\leq b$. Let $M_i:=\sum_{k=1}^{m_i}\alpha_{ik}$, $i=1,2$ and $M=M_1+M_2$.
\begin{Rem}
\label{re:normal_form}
 Observe that $P$ is an operator of order $M$ given in the normal form with respect to $\partial_{\Gamma_a,t}$, i.e. we may write it as
 \begin{equation}
 \label{eq:normal_form}
P(\partial_{\Gamma_a,t},\partial_z)=\partial_{\Gamma_a,t}^M-\sum_{l=1}^M P_l(\partial_z)\partial_{\Gamma_a,t}^{M-l}.
\end{equation}
\end{Rem}
\begin{Rem}
\label{re:H_w}
 Assume additionally that $|\lambda_{11}|\geq|\lambda_{12}|\geq...\geq|\lambda_{1m_1}|$. If we take $(j,\alpha)\in{\rm conv\,}\{\mathring{N}_s\}$ with $s=a-1$ then (see Theorem \ref{th:2})
 $$
 f_s(z)=P_s(z^{-1},1)=z^{-M}\prod_{k=1}^{m_1}(1-\lambda_{1k}z)^{\alpha_{1k}}.
 $$
 Observe that the condition (\ref{eq:H_w}) for $f(z)=f_s(z)z^j$ (equivalently, the Fredholm property for the operator $L(\partial_t,\partial_z)=P(\partial_t,\partial_z)\partial_t^{-j}\partial_z^{-\alpha}$ on $G^s_w$)
 holds if and only if there exists $\tilde{m}_1\in\{1,...,m_1\}$ such that
 $\sum_{k=1}^{\tilde{m}_1}\alpha_{1k}=M-j$, $|\lambda_{1,\tilde{m}_1}|>|\lambda_{1,\tilde{m}_1+1}|$ and $w\in(|\lambda_{1,\tilde{m}_1}|^{-1},|\lambda_{1,\tilde{m}_1+1}|^{-1})$.
\end{Rem}

\begin{Lem}
 \label{le:inhomogeneous}
  The Cauchy problem
 \begin{equation}
 \label{eq:inhomo_Q}
  \left\{
   \begin{array}{l}
     P(\partial_{\Gamma_a,t},\partial_z)w(t,z)=g(t,z)\in\Oo(D^2)\\
     \partial_{\Gamma_a,t}^lw(0,z)=0,\quad l=0,\dots,M-1.
   \end{array}
  \right.
  \end{equation}
  has the unique solution $w(t,z)\in\Oo(D^2)$
 
 Moreover, if we fix $d,\eta\in\RR$, $K\geq\frac{1}{a-b}$. and $p\in\NN$, and if we assume that $\arg\lambda_{1k}\in\{\eta+2n\pi/p\colon n\in\ZZ\}$ for $k=1,...,m_1$ and $g(t,z)\in\Oo^{K,aK}(\hat{S}_{d}\times\hat{S}_{(d+\eta+2n\pi/p)/a})$ for $n\in\ZZ$
then the solution $w(t,z)$ belongs to the same space
as $g(t,z)$.
\end{Lem}
\begin{proof}  
Since the operator $P(\partial_{\Gamma_a,t},\partial_z)$ satisfies (\ref{eq:normal_form}), there exists the unique formal
power series solution $w(t,z)\in\CC[[t,z]]$ of (\ref{eq:inhomo_Q}).

If $g(t,z)\in\Oo(D^2)$ then by \cite[Theorem 1]{Mic9}
also $w(t,z)\in\Oo(D^2)$ and
  \begin{equation}
  \label{eq:sum_w}
  w(t,z)=\sum_{j=1}^2\sum_{k=1}^{m_j}\sum_{l=1}^{\alpha_{jk}}w_{jkl}(t,z)
  \end{equation}
with $w_{jkl}\in\Oo_{1,1/\kappa}(D^2)$ satisfying
\begin{equation}
\label{eq:w_jkl}
\left\{
\begin{array}{l}
(\partial_{\Gamma_{a},t}-\lambda_{jk}(\partial_z))^lw_{jkl}
=g_{jkl}(t,z)\\
     \partial_{\Gamma_{a},t}^i w_{jkl}(0,z)=0,\ i=0,\dots,l-1,
   \end{array}
  \right.
  \end{equation}
where $g_{jkl}(t,z):=d_{jkl}(\partial_z)g(t,z)$ and $d_{jkl}(\zeta)$ is a holomorphic function of the variable $\xi=\zeta^{1/\kappa}$ and of moderate growth.

If additionally we assume that $g(t,z)\in\Oo^{K,aK}(\hat{S}_d\times\hat{S}_{(d+\eta+2n\pi/p)/a})$ for $n\in\ZZ$ then also $g_{jkl}(t,z)\in\Oo_{1,1/\kappa}^{K,aK}(\hat{S}_d\times\hat{S}_{(d+\eta+2n\pi/p)/a})$ for $n\in\ZZ$.

Moreover, by \cite[Proposition 9]{Mic9} the solution $w_{jkl}(t,z)$ of (\ref{eq:w_jkl}) has the
   integral representation
   \begin{equation}
     \label{eq:integral_repr}
     w_{jkl}(t,z)=\frac{-1}{2\kappa\pi i} \int_0^{t^{\frac{1}{a}}}\oint_{|w|=\varepsilon}^{\kappa} g_{jkl}(\tau^{a},w)\partial_{\tau}k_{jkl}(t,\tau,z,w)
  \,dw\,d\tau,
    \end{equation}
    where 
    \begin{equation}
    \label{eq:def_k}
    k_{jkl}(t,\tau,z,w):=\int_{r_0e^{i\theta}}^{\infty(\theta)}\lambda_{jk}^{-l}(\zeta)e_{a,l}\big((t^{\frac{1}{a}}-\tau)^{a}\lambda_{jk}(\zeta)\big)
    \mathbf{E}_{1/\kappa}(\zeta^{\frac{1}{\kappa}} z^{\frac{1}{\kappa}})e^{-\zeta w}\,d\zeta
    \end{equation}
    and  $e_{a,l}(x)$ is given by (see \cite[Lemma 3]{Mic9})
    $$
    e_{a,l}(x)=\sum_{n=l}^{\infty}\binom{n-1}{l-1}\frac{x^{n}}{\Gamma(1+an)}=\frac{1}{(l-1)!}x^{l}\Big(\frac{\mathbf{E}_a(x)-1}{x}\Big)^{(l-1)}.
    $$
    
Since the function $z\mapsto g_{jkl}(t,z)$ belongs to the space $\Oo_{1/\kappa}^{aK}(\hat{S}_{(d+\eta+2n\pi/p)/a})$ for $n\in\ZZ$, deforming the path of integration with respect to
$w$ in the integral representation (\ref{eq:integral_repr}) as in the proof of \cite[Lemma 4]{Mic8} we conclude that also
\begin{equation}
\label{eq:first_observation}
z\mapsto w_{jkl}(t,z)\in\Oo_{1/\kappa}^{aK}(\hat{S}_{(d+\eta+2n\pi/p)/a})\quad\textrm{for}\quad n\in\ZZ.
\end{equation}
Since in particular $g_{1kl}(t,z)\in\Oo_{1,1/\kappa}^{K,aK}(\hat{S}_d\times\hat{S}_{(d+\arg\lambda_{1k}+2n\pi)/a})$ for $n\in\ZZ$, by \cite[Theorem 2]{Mic9} also
$w_{1kl}(t,z)\in\Oo_{1,1/\kappa}^{K,aK}(\hat{S}_d\times\hat{S}_{(d+\arg\lambda_{1k}+2n\pi)/a})$ for $n\in\ZZ$.
Finally, by (\ref{eq:first_observation})
we conclude that 
$w_{1kl}(t,z)\in\Oo_{1,1/\kappa}^{K,aK}(\hat{S}_d\times\hat{S}_{(d+\eta+2n\pi/p)/a})$ for $n\in\ZZ$.

To prove the similar result for $w_{2kl}(t,z)$ observe that we may estimate the integrand of $k_{2kl}$ in (\ref{eq:def_k}) by
(see also the proof of \cite[Proposition 9]{Mic9})
$$
\big|\lambda_{2k}^{-l}(\zeta)e_{a,l}\big((t^{\frac{1}{a}}-\tau)^{a}\lambda_{2k}(\zeta)\big)
    \mathbf{E}_{1/\kappa}(\zeta^{\frac{1}{\kappa}} z^{\frac{1}{\kappa}})e^{-\zeta w}\big|\leq
Ae^{b_1|t^{1/a}-\tau||\zeta|^{q_k/a}}e^{b_2|\zeta||z|}e^{-b_3|\zeta||w|}.
$$
So, if $|z|$ is small relative to $|w|$ then for any $t\in\CC$ and $\tau\in[0,t^{1/a}]$ there exist $\tilde{A},\tilde{b}_1,\tilde{b}_2>0$ such that 
$$
|k_{2kl}(t,\tau,z,w)|
 \leq
    \int_{r_0}^{\infty}\tilde{A}e^{\tilde{b}_1x^{q_k/a}t^{1/a}-\tilde{b}_2x|w|}\,dx<
    \infty.
 $$
 Since the integrand in the above inequality is maximal for
 $x\sim |t|^{\frac{1}{a-q_k}}$ we conclude that the function
 $t\mapsto k_{2kl}(t,\tau,z,w)$ belongs to the space $\Oo_{1/\kappa}^{\frac{1}{a-q_k}}(\CC)$. On the other hand, since $K\geq\frac{1}{a-b}\geq\frac{1}{a-q_k}$ for $k=1,...,m_2$ and the function $t\mapsto g_{2kl}(t,z)$ belongs to the space $\Oo^{K}(\hat{S}_d)$, by (\ref{eq:integral_repr}) we get that also 
 $t\mapsto w_{2kl}(t,z)\in\Oo^{K}(\hat{S}_d)$.
Hence by (\ref{eq:first_observation}) we conclude that also
$w_{2kl}(t,z)\in\Oo_{1,1/\kappa}^{K,aK}(\hat{S}_d\times\hat{S}_{(d+\eta+2n\pi/p)/a})$ for $a\in\ZZ$.

Finally, since $w(t,z)\in\Oo(D^2)$, by (\ref{eq:sum_w}) we conclude that $w(t,z)
\in\Oo^{K,qK}((\hat{S}_d\times\hat{S}_{(d+\eta+2n\pi/p)/a})$ for $n\in\ZZ$.
\end{proof}

To prove the crucial result about the analytic continuation of the solutions of the Goursat problem we need the following special and refine version of \cite[Theorem 2]{Mic8} (see also \cite[Lemma 4]{Mic8})
\begin{Lem}
\label{le:crucial}
Let $\kappa\in\NN$, $d\in\RR$, $K>1$, $\beta\in\NN$, $\lambda\in\CC^*$,  and let
$\lambda(\zeta)\sim\lambda\zeta$ be an algebraic function, which
is also an analytic and invertible function of the variable $\zeta$ for $|\zeta|\geq r_0$ (for some $r_0>0$) with a simple pole at infinity.
Suppose that $u(t,z)\in\Oo(D^2)$ satisfies the equation
\begin{equation}
\label{eq:invertible}
  (\partial_{\Gamma_{1/\kappa},t}-\lambda(\partial_{\Gamma_{1/\kappa},z}))^{\beta}u(t,z)=0.
\end{equation}
Then for any $R<\infty$ we conclude that
$$\varphi_j(z):=\partial^j_{\Gamma_{1/\kappa},t}u(0,z)\in\Oo(\hat{S}_{d+\arg\lambda}\cap D_R)
$$
for $j=0,\dots,\beta-1$ if and only if
$$
\psi_n(t):=\partial^n_{\Gamma_{1/\kappa},z}u(t,0)\in\Oo(\hat{S}_d\cap D_{R/|\lambda|})\quad\textrm{for}\quad n=0,\dots,\beta-1.
$$
Moreover, for every $\varepsilon>0$ there exist constants $C_1,C_2<\infty$, which are independent of the Goursat data, such that the following conditions hold:
\begin{enumerate}
 \item[(a)] if $\varphi_j(z)\in\Oo^K(\hat{S}_{d+\arg\lambda})$ and
 there exist $A_1,B_1<\infty$ such that
 \begin{gather*}
 |\varphi_j(z)|\leq A_1e^{B_1|z|^{K}}\quad \text{for}
 \quad j=0,\dots,\beta-1
 \end{gather*}
 then $\psi_n(t)\in\Oo^K(\hat{S}_d)$ and
 \begin{gather*}
  |\psi_n(t)|\leq C_1A_1e^{B_1(1+\varepsilon)^K|\lambda|^K|t|^K}\quad\text{for}\quad n=0,\dots,\beta-1.
 \end{gather*}
\item[(b)] if there exist $A_2,B_2<\infty$ such that
$\psi_n(t)\in\Oo^K(\hat{S}_d)$ and
\begin{gather*}
 |\psi_n(t)|\leq A_2e^{B_2|\lambda|^K|t|^K} \quad \text{for} \quad n=0,\dots,\beta-1
 \end{gather*} 
then $\varphi_j(z)\in\Oo^K(\hat{S}_{d+\arg\lambda})$
 \begin{gather*}
 |\varphi_j(z)|\leq C_2A_2e^{B_2(1+\varepsilon)^K|z|^{K}}\quad \text{for}
 \quad j=0,\dots,\beta-1.
 \end{gather*}
\end{enumerate}
\end{Lem}

\begin{proof}
 ($\Longrightarrow$)
 By the principle of superposition of solutions of linear equations we may assume that
 $\varphi_j(z)=\partial^j_{\Gamma_{1/\kappa},t}u(0,z)=0$ for $j=0,\dots,\beta-2$ and
 $\varphi_{\beta-1}(z)=
 \partial_{\Gamma_{1/\kappa},t}^{\beta-1} u(0,z)=\lambda^{\beta-1}(\partial_{\Gamma_{1/\kappa},z})\varphi(z)$
 for some $\varphi(z)\in\Oo(\hat{S}_{d+\arg\lambda}\cap D_R)$.

By \cite[Lemma 3]{Mic8} we get
\begin{gather*}
\psi_n(t)=\partial^n_{\Gamma_{1/\kappa},z}u(t,0)=\frac{t^{\beta-1}}{(\beta-1)!}\partial_t^{\beta-1}
   \frac{1}{2\pi i}\oint_{|w|=\delta}\varphi(w)
   \int_{r_0e^{i\theta}}^{\infty(\theta)}\mathbf{E}_{1/\kappa}(t\lambda(\zeta))
    \zeta^{n}\kappa\zeta^{\kappa-1}w^{\kappa-1}e^{-\zeta^{\kappa}w^{\kappa}}d\zeta\,dw,
\end{gather*}
  where $\theta\in (-\arg w-\frac{\pi}{2\kappa}, -\arg w + \frac{\pi}{2\kappa})$.
  
  Since the Mittag-Leffler function $\mathbf{E}_{\alpha}(z)$
  is entire function of exponential growth of order $1/\alpha$ and of type $1$ (see \cite[Appendix B.4]{B2}), taking $\theta=-\arg w$ we conclude that for every $\varepsilon>0$ we may find $A<\infty$ such that
\begin{gather*}
 \Big|\int_{r_0e^{i\theta}}^{\infty(\theta)}\mathbf{E}_{1/\kappa}(t\lambda(\zeta))
    \zeta^{n}\kappa\zeta^{\kappa-1}w^{\kappa-1}e^{-\zeta^{\kappa}w^{\kappa}}d\zeta\Big| \le
    \int_{r_0}^{\infty}Ae^{((1+\varepsilon)|t|^{\kappa}|\lambda|^{\kappa}-|w|^{\kappa})s^{\kappa}}s^{\kappa-1}w^{\kappa-1}s^{n}\,ds<\infty
\end{gather*}
for $(1+\varepsilon)|t|^{\kappa}|\lambda|^{\kappa}<|w|^{\kappa}$. Since $\varepsilon > 0$ is arbitrary, it follows that the integral is
convergent for $|t|<|w|/|\lambda|$, hence $\psi_n(t)\in\Oo(D)$. Furthermore,
we deforme the path of integration with respect to $w$ to $\gamma_R$ as in
\cite[Lemma 4]{Mic8}. Here $\gamma_R$ is a contour became from the circle $|w|=\delta$ by the deforming the arc $\{w\colon |w|=\delta,\ |\arg w-d-\arg\lambda|<\beta/2\}$ (for some $\beta$ in the opening of $\hat{S}_{d+\arg\lambda}$) into the path along the ray $\arg w = d+\arg\lambda-\beta/2$ to a point with modulus $R$ (which can be
chosen arbitrarily large), then along the circle $|w| = R$ to the ray $\arg w = d + \arg\lambda + \beta/2$ and back along this
ray to the original circle. Using this deformed contour 
we conclude that $\psi_n(t)$ is analytically continued to the set $\hat{S}_d\cap D_{R/|\lambda|}$ under condition that $\varphi(z)\in\Oo(\hat{S}_{d+\arg\lambda}\cap D_R)$. 

Repeating the estimations from \cite[Lemma 4]{Mic8}  we conclude that for any $\varepsilon>0$ there exists $\tilde{C}_1,C_1<\infty$ independent of $\varphi(z)$ such that
\[
 |\psi_n(t)|\leq \tilde{C}_1 \oint_{\gamma_R}|\varphi(w)|\,d|w|\leq  C_1A_1e^{B_1(1+\varepsilon)^K|\lambda|^K|t|^K}\quad\text{for}\quad n=0,\dots,\beta-1,
\]
where $R=(1+\varepsilon/2)|t||\lambda|$.
\medskip\par
($\Longleftarrow$)
To prove the lemma in the opposite side, we observe that by \cite[Lemmas 6 and 7]{Mic7},
if $u$ satisfies the equation (\ref{eq:invertible}) then $u$ is also a solution
of the equation
\begin{equation}
\label{eq:invertible_2}
(\partial_{\Gamma_{1/\kappa},z}-\lambda^{-1}(\partial_{\Gamma_{1/\kappa,t}}))^{\beta}u=0.
\end{equation}
Since $u$ satisfies (\ref{eq:invertible_2}) and $\lambda^{-1}(\zeta)\sim\lambda^{-1}\zeta$, swapping the role of variables and repeating the first part of the proof we
get the assertion.
\end{proof}

Now we are ready to prove the
next two crucial lemmas about the analytic continuation of the solutions of the Goursat problem
\begin{Lem}
\label{le:crucial_2}
Let $s=a-1$, $K\geq\frac{a\kappa}{a-b}$ and $(j,\alpha)\in{\rm conv\,}\{\mathring{N}_s\}$.
Assume that the operator
$L(\partial_t,\partial_z)=P(\partial_t,\partial_z)\partial_t^{-j}\partial_z^{-\alpha}$ has Fredholm property on $G^s_w$ and is bijective.

Then the Goursat problem
 \begin{equation}
  \label{eq:Goursat_Q}
  \left\{
   \begin{array}{l}
    Q(\partial_{\Gamma_{1/\kappa},t},\partial_{\Gamma_{1/\kappa},z})w(t,z)=0\\
     w(t,z)-v(t^{a\kappa},z^{\kappa})=O(t^{\kappa j} z^{\kappa\alpha})\quad\textrm{for some}\quad v(t,z)\in\Oo(D^2)
   \end{array}
  \right.
\end{equation}
is uniquely solvable in the space of holomorphic functions $\Oo(D^2)$,
where 
\begin{equation}
\label{eq:Q}
Q(\partial_{\Gamma_{1/\kappa},t},\partial_{\Gamma_{1/\kappa},z}):=
P(\partial^{a\kappa}_{\Gamma_{1/\kappa},t},\partial^{\kappa}_{\Gamma_{1/\kappa},z}).
\end{equation}

Moreover, we may factorise the solution $w$ as
$$
w(t,z)=\sum_{j=1}^2\sum_{k=1}^{m_j}\sum_{l=1}^{\alpha_{jk}}\sum_{n=0}^{a\kappa-1}w_{jkln}(t,z),
$$
where $w_{jkln}(t,z)$ satisfies the equation
\begin{equation*}
 \left\{
  \begin{array}{l}
    (\partial_{\Gamma_{1/\kappa},t}-\lambda_{jkn}(\partial_{\Gamma_{1/\kappa},z}))^lw_{jkln}(t,z)=0\\
     \partial_{\Gamma_{1/\kappa},t}^j w_{jkln}(0,z)=0,\ j=0,\dots,l-2\\
     \partial_{\Gamma_{1/\kappa},t}^{l-1} w_{jkln}(0,z)=(l-1)!\varphi_{jkln}(z)\in\Oo_{1/\kappa}(D)
   \end{array}
  \right.
\end{equation*}
with $\lambda_{jkn}(\zeta):=e^{\frac{2n\pi i}{a\kappa}}\lambda_{jk}^{1/a\kappa}(\zeta^{\kappa})$.
In particular $\lambda_{1kn}(\zeta)$ is holomorphic and invertible for sufficiently large $|\zeta|$ with a simple pole at infinity. Additionally $w_{1kln}(t,z)\in\Oo(D^2)$ and $w_{2kln}\in\Oo^K_{1,1/\kappa}(\CC\times D)$.
\end{Lem}

\begin{proof}
By Theorem \ref{th:1}, for given $\tilde{v}(t,z)\in\Oo[[t]]_s$ there exists the unique
solution $u(t,z)\in\Oo[[t]]_s$ of the Goursat problem
\begin{equation}
  \label{eq:Goursat_u}
  \left\{
   \begin{array}{l}
    P(\partial_{t},\partial_{z})u(t,z)=0\\
     u(t,z)-\tilde{v}(t,z)=O(t^j z^{\alpha}).
   \end{array}
  \right.
\end{equation}

We will show that $w(t,z):=(\Bo_{s} u)(t^{a\kappa},z^{\kappa})\in\Oo(D^2)$
 satisfies the Goursat problem (\ref{eq:Goursat_Q}) with $v(t,z):=\Bo_{s}\tilde{v}(t,z)\in\Oo(D^2)$.
 
We prove it in a similar way to \cite[Lemma 3]{Mic8}. 
Let $\tilde{w}(t,z):=\Bo_{s} u(t,z)$. Since
$\Bo_{s}\partial_t u(t,z)=\partial_{\Gamma_a,t}\Bo_{s} u(t,z)$ (see Remark \ref{re:5}),
we deduce that $\tilde{w}(t,z)$ satisfies the Goursat problem
\begin{equation*}
  \left\{
   \begin{array}{l}
    P(\partial_{\Gamma_a,t},\partial_z)\tilde{w}(t,z)=0\\
     \tilde{w}(t,z)-v(t,z)=O(t^j z^{\alpha}).
   \end{array}
  \right.
\end{equation*}

Next, observe that 
\begin{gather*}
(\partial_{\Gamma_a,t}\tilde{w})(t^{a\kappa},z^{\kappa})=\partial^{a\kappa}_{\Gamma_{1/\kappa},t}(\tilde{w}(t^{a\kappa},z^{\kappa}))\quad\text{and}\quad
(\partial_{z}\tilde{w})(t^{a\kappa},z^{\kappa})=\partial^{\kappa}_{\Gamma_{1/\kappa},z}(\tilde{w}(t^{a\kappa},z^{\kappa})).
\end{gather*}

Hence $w(t,z):=\tilde{w}(t^{a\kappa},z^{\kappa})$ satisfies (\ref{eq:Goursat_Q}).

Since $u(t,z)$ is a solution of (\ref{eq:Goursat_u}), by \cite[Theorem 1]{Mic8}
$$
u(t,z)=\sum_{j=1}^2\sum_{k=1}^{m_j}\sum_{l=1}^{\alpha_{jk}}u_{jkl}(t,z),
$$
where $u_{jkl}(t,z)$ satisfies
\begin{equation*}
 \left\{
  \begin{array}{l}
    (\partial_{t}-\lambda_{jk}(\partial_{z}))^lu_{jkl}(t,z)=0\\
     \partial^n_{t} u_{jkl}(0,z)=0,\ n=0,\dots,l-2\\
     \partial^{l-1}_{t} u_{jkl}(0,z)=(l-1)!\varphi_{jkl}(z)
   \end{array}
  \right.
\end{equation*}
for some $\varphi_{jkl}(t,z)\in\Oo_{1/\kappa}(D)$.
By \cite[Theorem 1]{Mic8}
we also conclude that $u_{1kl}(t,z)\in\Oo_{1/\kappa}[[t]]_s$ and
$u_{2kl}(t,z)\in\Oo_{1/\kappa}[[t]]_{q_k-1}\subseteq\Oo_{1/\kappa}[[t]]_{b-1}.$ 

Hence $w_{jkl}(t,z):=(\Bo_{s}u_{jkl})(t^{\alpha\kappa},z^{\kappa})$ satisfies the equation
\begin{equation*}
    (\partial_{\Gamma_{1/\kappa},t}^{a\kappa}-\lambda_{jk}(\partial_{\Gamma_{1/\kappa},z}^{\kappa}))^lw_{jkl}(t,z)=0.
\end{equation*}
Moreover $w_{1kl}(t,z)\in\Oo(D^2)$
and $w_{2kl}(t,z)\in\Oo^K(\CC\times D)$.

On the other hand, observe that
$$
Q(\partial_{\Gamma_{1/\kappa},t},\partial_{\Gamma_{1/\kappa},z})=
\prod_{j=1}^2\prod_{k=1}^{m_j}\prod_{n=0}^{a\kappa-1}(\partial_{\Gamma_{1/\kappa},t}-
\lambda_{jkn}(\partial_{\Gamma_{1/\kappa},z}))^{\alpha_{jk}},$$
where
$\lambda_{jkn}(\partial_{\Gamma_{1/\kappa},z}):=e^{\frac{2n\pi i}{a\kappa}}\lambda_{jk}^{1/\kappa\alpha}(\partial_{\Gamma_{1/\kappa},z}^{\kappa})$.

It means that 
$\lambda_{1kn}(\zeta)\sim e^{\frac{2n\pi i}{a\kappa}}\lambda_{1k}^{1/\kappa\alpha}\zeta$ and
$\lambda_{2kn}(\zeta)\sim e^{\frac{2n\pi i}{a\kappa}}\lambda_{2k}^{1/\kappa\alpha}
\zeta^{q_k/a}$.

So, since $w$ is a solution
of (\ref{eq:Goursat_Q}), once again by \cite[Theorem 1]{Mic8}, we conclude that
$$
w(t,z)=\sum_{j=1}^2\sum_{k=1}^{m_j}\sum_{l=1}^{\alpha_{jk}}\sum_{n=0}^{a\kappa-1}w_{jkln}(t,z),
$$
where $w_{jkln}(t,z)$ satisfies the equation
\begin{equation*}
 \left\{
  \begin{array}{l}
    (\partial_{\Gamma_{1/\kappa},t}-\lambda_{jkn}(\partial_{\Gamma_{1/\kappa},z}))^lw_{jkln}(t,z)=0\\
     \partial_{\Gamma_{1/\kappa},t}^j w_{jkln}(0,z)=0,\ j=0,\dots,l-2\\
     \partial_{\Gamma_{1/\kappa},t}^{l-1} w_{jkln}(0,z)=(l-1)!\varphi_{jkln}(z)\in\Oo_{1/\kappa}(D).
   \end{array}
  \right.
\end{equation*}
Moreover $w_{1kln}(t,z)\in\Oo_{1,1/\kappa}(D^2)$ and $w_{2kln}(t,z)\in\Oo_{1,1/\kappa}^K(\CC\times D)$.

Additionally by Lemma \ref{le:inverse}
the function $\lambda_{1kn}(\zeta)$ is holomorphic and invertible in $\zeta$ for sufficiently large $\zeta$ with a simple pole at infinity.
So, since $\sum_{n=0}^{a\kappa-1}w_{1kln}(t,z)=w_{1kl}(t,z)\in\Oo(D^2)$ by
\cite[Theorem 1]{Mic8} (see also 
\cite[Theorem 1 and formula (17)]{Mic7})
we conclude that $w_{1kln}(t,z)\in\Oo(D^2)$.
\end{proof}

In the next lemma for abbreviation we change the notation connected with the factorisation of the operator $Q$.
\begin{Lem}
\label{le:crucial_3}
Suppose that the constants $a,b,\kappa$ satisfy the conditions of Lemma \ref{le:crucial_2} and the operator $Q$ is defined by (\ref{eq:Q}). 

Then we may 
factorise
the operator $Q$ as
$$
Q(\partial_{\Gamma_{1/\kappa},t},\partial_{\Gamma_{1/\kappa},z})=
\prod_{j=1}^2\prod_{k=1}^{m_j}(\partial_{\Gamma_{1/\kappa},t}-\lambda_{jk}(\partial_{\Gamma_{1/\kappa},z}))^{\alpha_{jk}}
$$
for some constants $m_j$, $\alpha_{jk}$ and for some algebraic functions $\lambda_{jk}(\zeta)$ which are holomorphic functions of the variable $\xi=\zeta^{\kappa}$ for $|\zeta|\geq r_0$ (for some $r_0>0$), such that $\lambda_{2k}(\zeta)\sim \lambda_{2k}\zeta^{q_k/a}$, $q_k\leq b$, and
$\lambda_{1k}(\zeta)\sim \lambda_{1k}\zeta$ is additionally an analytic and invertible function of the variable $\zeta$ for $|\zeta|\geq r_0$  with a simple pole at infinity.

Moreover, if we additionally assume that $d,\eta\in\RR$, $K\geq\frac{a\kappa}{a-b}$, $p\in\NN$ and  $\arg\lambda_{1k}\in\{(\eta+2n\pi/p)/\kappa\colon n\in\ZZ\}$ for $k=1,...,m_1$
then the solution $w(t,z)$ of the Goursat problem
 \begin{equation}
  \label{eq:Goursat_Q2}
  \left\{
   \begin{array}{l}
    Q(\partial_{\Gamma_{1/\kappa},t},\partial_{\Gamma_{1/\kappa},z})w(t,z)=0\\
     w(t,z)-v(t,z)=O(t^{j\kappa} z^{\kappa\alpha})\quad\textrm{for some}\quad v(t,z)\in\Oo(D^2)
   \end{array}
  \right.
\end{equation}
belongs to the space $\Oo^{K,K}(\hat{S}_{(d+2l\pi)/\kappa}\times\hat{S}_{(d+\eta+2n\pi/p)/\kappa})$ for $n,l\in\ZZ$ if and only if the Goursat data $v(t,z)$ belong to the same space. 
\end{Lem}

\begin{proof}
Observe that the first part of the lemma is given immediately by Lemma \ref{le:crucial_2} with the changed meaning of the constants $m_j,\alpha_{jk}$, and with algebraic functions $\lambda_{jkn}(\zeta)$
replaced by $\lambda_{jk}(\zeta)$.

To prove the second part we take the constants $M_1,M_2,M,J\in\NN$ defined by
$M_j:=\sum_{k=1}^{m_j}\alpha_{jk}$ ($j=1,2$), $M:=M_1+M_2$ and $J:=j\kappa$.

Since the Goursat problem (\ref{eq:Goursat_Q2}) satisfies the condition (\ref{eq:H_w}) 
there exist $\tilde{m}_1\in\{1,...,m_1\}$ such that 
$\sum_{k=1}^{\tilde{m}_1}\alpha_{1k}=M-J$.

Moreover, by Remark \ref{re:H_w} we may assume that
$|\lambda_{1,1}|\geq...\geq|\lambda_{1,\tilde{m}_1}|>|\lambda_{1,\tilde{m}_1+1}|\geq...\geq |\lambda_{1,m_1}|$.

If $w(t,z)\in\Oo^{K,K}(\hat{S}_{(d+2l\pi)/\kappa}\times\hat{S}_{(d+\eta+2n\pi/p)/\kappa})$ for $l,n\in\ZZ$ then, directly by the definition, $v(t,z)$  belongs to the same space. For this reason it is sufficient to prove the implication in the opposite side.

So we assume that  $v(t,z)\in\Oo^{K,K}(\hat{S}_{(d+2l\pi)/\kappa}\times\hat{S}_{(d+\eta+2n\pi/p)/\kappa})$ for $l,n\in\ZZ$. We will show that $w(t,z)$ belongs to the same space.

Since the Goursat problem (\ref{eq:Goursat_Q2}) is uniquely solvable in the space of holomorphic functions $\Oo(D^2)$, its solution $w(t,z)$ satisfies also the Cauchy problem
\begin{equation*}
  \left\{
   \begin{array}{l}
    Q(\partial_{\Gamma_{1/\kappa},t},\partial_{\Gamma_{1/\kappa},z})w(t,z)=0\\
     \partial^j_{\Gamma_{1/\kappa},t}w(0,z)=\varphi_j(z)\quad\text{for}\quad j=0,...,M-1,
   \end{array}
  \right.
\end{equation*}
for some $\varphi_j(z)\in\Oo(D)$.

Hence, by \cite[Theorem 1]{Mic8}
$$
w(t,z)=\sum_{j=1}^2\sum_{k=1}^{m_j}\sum_{l=1}^{\alpha_{jk}}w_{jkl}(t,z),
$$
where $w_{jkl}(t,z)$ satisfies
\begin{equation*}
 \left\{
  \begin{array}{l}
    (\partial_{\Gamma_{1/\kappa},t}-\lambda_{jk}(\partial_{\Gamma_{1/\kappa},z}))^lw_{jkl}(t,z)=0\\
     \partial^n_{\Gamma_{1/\kappa},t} w_{jkl}(0,z)=0,\ n=0,\dots,l-2\\
     \partial^{l-1}_{\Gamma_{1/\kappa},t} w_{jkl}(0,z)=(l-1)!\varphi_{jkl}(z)\in\Oo_{1/\kappa}(D),
   \end{array}
  \right.
\end{equation*}
and $w_{1kl}(t,z)\in\Oo_{1,1/\kappa}(D^2)$, $w_{2kl}(t,z)\in\Oo_{1,1/\kappa}^{\frac{a\kappa}{a-q_k}}(\CC\times D)$.

By Lemma \ref{le:crucial_2} 
we see that $\varphi_{1kl}(z)\in\Oo(D)$,
hence $w_{1kl}(t,z)\in\Oo(D^2)$.
Since $q_k\leq b$, we also conclude that $w_{2kl}(t,z)\in\Oo_{1,1/\kappa}^{K}(\CC\times D)$.

Observe that $w_{1k}(t,z):=\sum_{l=1}^{\alpha_{1k}}w_{1kl}(t,z)$ satisfies the equation
$$
(\partial_{\Gamma_{1/\kappa},t}-\lambda_{1k}(\partial_{\Gamma_{1/\kappa},z}))^{\alpha_{1k}}w_{1k}(t,z)=0.
$$
Hence, since $\lambda_{1k}(\zeta)\sim\lambda_{1k}\zeta$ is invertible and holomorphic with respect to $\zeta$ for $|\zeta|\geq r_0$ with a simple pole at infinity, by \cite[Lemma 7]{Mic7}, $w_{1k}(t,z)$ satisfies also the equation
$$
(\partial_{\Gamma_{1/\kappa},z}-\lambda_{1k}^{-1}(\partial_{\Gamma_{1/\kappa},t}))^{\alpha_{1k}}w_{1k}(t,z)=0.
$$
So, by linearity we may write $w_{1k}(t,z)=\sum_{l=1}^{\alpha_{1k}}\tilde{w}_{1kl}(t,z)$, where $\tilde{w}_{1kl}(t,z)$
is a solution of the Cauchy problem with replaced variables
\begin{equation*}
 \left\{
  \begin{array}{l}
    (\partial_{\Gamma_{1/\kappa},z}-\lambda_{1k}^{-1}(\partial_{\Gamma_{1/\kappa},t}))^l\tilde{w}_{1kl}(t,z)=0\\
     \partial^n_{\Gamma_{1/\kappa},z} \tilde{w}_{1kl}(t,0)=0,\ n=0,\dots,l-2\\
     \partial^{l-1}_{\Gamma_{1/\kappa},z} \tilde{w}_{1kl}(t,0)=(l-1)!\psi_{1kl}(t)
   \end{array}
  \right.
\end{equation*}
for some $\psi_{1kl}(t)\in\Oo(D)$.

We may write our Goursat problem as follows:
$$\partial^n_{\Gamma_{1/\kappa},t}w(0,z)=\partial^n_{\Gamma_{1/\kappa},t}v(0,z)
=:\Phi_{n}(z)\in\Oo^{K}\big(\bigcup_{k\in\ZZ}\hat{S}_{(d+\eta+2k\pi/p)/\kappa}\big),\quad n=0,...,J-1$$
$$\partial^n_{\Gamma_{1/\kappa},z}w(t,0)=\partial^n_{\Gamma_{1/\kappa},z}v(t,0)
=:\Psi_{n}(t)\in\Oo^{K}\big(\bigcup_{k\in\ZZ}\hat{S}_{(d+2k\pi)/\kappa}\big)),\quad n=0,...,M-J-1.$$

By \cite[Lemma 2]{Mic7}
\begin{equation*}
 w_{jkl}(t,z)=\sum_{m=l-1}^{\infty}\frac{m!}{(m-l+1)!}\lambda_{jk}(\partial_{\Gamma_{1/\kappa},z})^{m-l+1}\frac{\varphi_{jkl}(z)}{\Gamma(1+m/\kappa)}t^m
\end{equation*}
and similarly
\begin{equation*}
 \tilde{w}_{1kl}(t,z)=\sum_{m=l-1}^{\infty}\frac{m!}{(m-l+1)!}\lambda^{-1}_{1k}(\partial_{\Gamma_{1/\kappa},t})^{m-l+1}\frac{\psi_{1kl}(t)}{\Gamma(1+m/\kappa)}z^m.
\end{equation*}
It means that the Goursat data generate the following system of equations
\begin{gather*}
\left\{
  \begin{array}{l}
   \sum_{j=1}^2\sum_{k=1}^{m_j}\sum_{l=1}^{\max\{\alpha_{jk},n\}}
   \frac{(n-1)!}{(n-l)!}\lambda_{jk}(\partial_{\Gamma_{1/\kappa},z})^{n-l}\varphi_{jkl}(z)=\Phi_{n-1}(z),\ n=1,...,J\\
   \sum_{k=1}^{m_1}\sum_{l=1}^{\max\{\alpha_{1k},n\}}
   \frac{(n-1)!}{(n-l)!}\lambda_{1k}^{-1}(\partial_{\Gamma_{1/\kappa},t})^{n-l}\psi_{1kl}(t)=
   \tilde{\Psi}_{n-1}(t),\ n=1,...,M-J,\\
  \end{array}
  \right.
\end{gather*}
where
$$\tilde{\Psi}_{n}(t):=\Psi_n(t)-\sum_{k=1}^{m_2}\sum_{l=1}^{\alpha_{2k}}\partial^n_{\Gamma_{1/\kappa},z}w_{2kl}(t,0)\in\Oo^{K}\big(\bigcup_{k\in\ZZ}\hat{S}_{(d+2k\pi)/\kappa}\big).$$ 

Let us write this system in the vector notation
\begin{gather*}
 \left\{
  \begin{array}{l}
  \mathbf{\Lambda_{11}}(\partial_{\Gamma_{1/\kappa},z})\mathbf{\varphi_{11}}(z)+\mathbf{\Lambda_{12}}(\partial_{\Gamma_{1/\kappa},z})\mathbf{\varphi_{12}}(z)+\mathbf{\Lambda_{2}}(\partial_{\Gamma_{1/\kappa},z})\mathbf{\varphi_{2}}(z)
  =\mathbf{\Phi}(z)\\
  \mathbf{\tilde{\Lambda}_{11}}(\partial_{\Gamma_{1/\kappa},t})\mathbf{\psi_{11}}(t)
  +\mathbf{\tilde{\Lambda}_{12}}(\partial_{\Gamma_{1/\kappa},t})\mathbf{\psi_{12}}(t)=\mathbf{\tilde{\Psi}}(t),
  \end{array}
  \right.
\end{gather*}
where
$$\mathbf{\Phi}(z):=\big(\Phi_0(z),...,\Phi_{J-1}(z)\big)^T\in\Big(\Oo^{K}\big(\bigcup_{k\in\ZZ}\hat{S}_{(d+\eta+2k\pi/p)/\kappa}\big)\Big)^{J},$$
$$\mathbf{\tilde{\Psi}}(t):=
\big(\tilde{\Psi}_0(t),...,\tilde{\Psi}_{M-J-1}(t)\big)^T\in\Big(\Oo^{K}\big(\bigcup_{k\in\ZZ}\hat{S}_{(d+2n\pi)/\kappa}\big)\Big)^{M-J}.$$
We also denote
$$
\mathbf{\varphi_{11}}(z):=\big(\varphi_{111}(z),...,\varphi_{11\alpha_{11}}(z),
...,\varphi_{1\tilde{m}_11}(z),...,\varphi_{1\tilde{m}_1\alpha_{1\tilde{m}_1}}(z)\big)^T\in\big(\Oo(D)\big)^{M-J},
$$
$$
\mathbf{\varphi_{12}}(z):=\big(\varphi_{1,\tilde{m}_1+1,1}(z),...,\varphi_{1,\tilde{m}_1+1,\alpha_{1,\tilde{m}_1+1}}(z),...,\varphi_{1m_11}(z),...,\varphi_{1m_1\alpha_{1m_1}}(z)\big)^T\in\big(\Oo(D)\big)^{J-M_2},
$$
$$
\mathbf{\varphi_{2}}(z):=\big(\varphi_{211}(z),...,\varphi_{21\alpha_{21}}(z),
...,\varphi_{2m_21}(z),...,\varphi_{2m_2\alpha_{2m_2}}(z)\big)^T\in\big(\Oo_{1/\kappa}(D)\big)^{M_2},
$$
$\mathbf{\Lambda_1}(\zeta):=\big(a_{ij}(\zeta)\big)^{i=1,...,J}_{j=1,...,M_1}$,
where
$$
a_{ij}(\zeta):=\left\{
  \begin{array}{ll}
   \frac{(i-1)!}{(i-l)!}\lambda_{1m_0}(\zeta)^{i-l} & i\geq l\\
   0 & i<l.
  \end{array}
  \right.
$$
with $l$ and $m_0$ generated uniquely by $j$ in a such way that
\begin{equation}
\label{eq:l_and_m_0}
\sum_{k=1}^{m_0-1}\alpha_{1k}<j\le\sum_{k=1}^{m_0}\alpha_{1k}\quad\textrm{and}\quad
l:=j-\sum_{k=1}^{m_0-1}\alpha_{1k};
\end{equation}
$\mathbf{\Lambda_1}(\zeta):=\big(\mathbf{\Lambda_{11}}(\zeta),\mathbf{\Lambda_{12}}(\zeta)\big)$ and $\mathbf{\Lambda_{11}}(\zeta):=\big(a_{ij}(\zeta)\big)^{i=1,...,J}_{j=1,...,M-J}$,
$\mathbf{\Lambda_{12}}(\zeta):=\big(a_{ij}(\zeta)\big)^{i=1,...,J}_{j=M-J+1,...,M_1}$.

Similarly 
$\mathbf{\Lambda_2}(\zeta):=\big(b_{ij}(\zeta)\big)^{i=1,...,J}_{j=1,...,M_2}$,
where
$$
b_{ij}(\zeta):=\left\{
  \begin{array}{ll}
   \frac{(i-1)!}{(i-l)!}\lambda_{2m_0}(\zeta)^{i-l} & i\geq l\\
   0 & i<l.
  \end{array}
  \right.
$$
with $l$ and $m_0$ defined uniquely by $j$
as in (\ref{eq:l_and_m_0}), but with $\alpha_{1k}$ replaced by $\alpha_{2k}$.

Analogously,
$$
\mathbf{\psi_{11}}(t):=\big(\psi_{111}(t),...,\psi_{11\alpha_{11}}(t),
...,\psi_{1\tilde{m}_11}(t),...,\psi_{1\tilde{m}_1\alpha_{1\tilde{m}_1}}(t)\big)^T\in\big(\Oo(D)\big)^{M-J},
$$
$$
\mathbf{\psi_{12}}(t):=\big(\psi_{1,\tilde{m}_1+1,1}(t),...,\psi_{1,\tilde{m}_1+1,\alpha_{1,\tilde{m}_1+1}}(t),...,\psi_{1m_11}(t),...,\psi_{1m_1\alpha_{1m_1}}(t)\big)^T\in\big(\Oo(D)\big)^{J-M_2},
$$
$\mathbf{\tilde{\Lambda}_1}(\zeta):=\big(\tilde{a}_{ij}(\zeta)\big)^{i=1,...,J}_{j=1,...,M_1}$,
where
$$
\tilde{a}_{ij}(\zeta):=\left\{
  \begin{array}{ll}
   \frac{(i-1)!}{(i-l)!}\lambda^{-1}_{1m_0}(\zeta)^{i-l} & i\geq l\\
   0 & i<l.
  \end{array}
  \right.
$$
with $l=l(j)$ and $m_0=m_0(j)$ defined by (\ref{eq:l_and_m_0}),
$\mathbf{\tilde{\Lambda}_1}(\zeta):=\big(\mathbf{\tilde{\Lambda}_{11}}(\zeta),\mathbf{\tilde{\Lambda}_{12}}(\zeta)\big)$ and\\
$\mathbf{\tilde{\Lambda}_{11}}(\zeta):=\big(\tilde{a}_{ij}(\zeta)\big)^{i=1,...,J}_{j=1,...,M-J}$,
$\mathbf{\tilde{\Lambda}_{12}}(\zeta):=\big(\tilde{a}_{ij}(\zeta)\big)^{i=1,...,J}_{j=M-J+1,...,M_1}$.

Since $\big(\mathbf{\Lambda_{12}}(\zeta),\mathbf{\Lambda_2}(\zeta)\big)$
and $\mathbf{\tilde{\Lambda}_{11}}(\zeta)$ are confluent Vandermonde matrices,
they are invertible for sufficiently large $|\zeta|$, say $|\zeta|>r_0$. Hence we conclude that
\begin{gather}
\label{eq:system_solved}
 \left\{
  \begin{array}{l}
   \begin{pmatrix}
    \mathbf{\varphi_{12}}(z) \\
    \mathbf{\varphi_2}(z)
   \end{pmatrix}
   = \Big(\mathbf{\Lambda_{12}}(\partial_{\Gamma_{1/\kappa},z}),\mathbf{\Lambda_2}(\partial_{\Gamma_{1/\kappa},z})\Big)^{-1}\Big[\mathbf{\Phi}(z)-\mathbf{\Lambda}_{11}(\partial_{\Gamma_{1/\kappa},z})\mathbf{\varphi_{11}}(z)\Big]\\
   \mathbf{\psi_{11}}(t)=\mathbf{\tilde{\Lambda}_{11}}(\partial_{\Gamma_{1/\kappa},t})^{-1}\Big[\mathbf{\tilde{\Psi}}(t)-\mathbf{\tilde{\Lambda}_{12}}(\partial_{\Gamma_{1/\kappa},t})\mathbf{\psi_{12}}(t)\Big]
  \end{array}
  \right.
\end{gather}

Without loss of generality we may assume that there exists $R>0$
such that 
$$\mathbf{\varphi_1}(z)=(\mathbf{\varphi_{11}}(z),\mathbf{\varphi_{12}}(z))\in(\Oo(D_R))^{M_1}.$$
Since $|\lambda_{1,\tilde{m}_1+1}|\geq|\lambda_{1,k}|$ for $k=\tilde{m}_1+1,...,m_1$,
by Lemma \ref{le:crucial} we conclude that\\
$\mathbf{\psi_{12}}(t)\in(\Oo(D_{R/|\lambda_{1,\tilde{m}_1+1}|}))^{J-M_2}$. By the second equation in (\ref{eq:system_solved}) also $\mathbf{\psi_{11}}(t)\in(\Oo(\hat{S}_{(d+\eta+2n\pi/p)/\kappa}\cap D_{R/|\lambda_{1,\tilde{m}_1+1}|}))^{M-J}$ for $n\in\ZZ$.

Since $|\lambda_{1,\tilde{m}_1}|\leq|\lambda_{1,k}|$ for $k=1,...,\tilde{m}_1$,
applying once again Lemma \ref{le:crucial} we get that
$$
\mathbf{\varphi_{11}}(z)\in
(\Oo(\hat{S}_{(d+\eta+2n\pi/p)/\kappa}\cap D_{R|\lambda_{1,\tilde{m}_1}|/|\lambda_{1,\tilde{m}_1+1}|}))^{M-J}\quad\textrm{for}\quad n\in\ZZ.
$$
Moreover, by the first equation in (\ref{eq:system_solved}) also
$$
\mathbf{\varphi_{12}}(z)\in
(\Oo(\hat{S}_{(d+\eta+2n\pi/p)/\kappa}\cap D_{R|\lambda_{1,\tilde{m}_1}|/|\lambda_{1,\tilde{m}_1+1}|}))^{J-M_2}\quad\textrm{for}\quad n\in\ZZ.
$$
If we repeat this procedure $N$ times we get 
$$
\mathbf{\varphi_{1}}(z)\in
(\Oo(\hat{S}_{d+\eta+2n\pi/p}\cap D_{R(|\lambda_{1,\tilde{m}_1}|/|\lambda_{1,\tilde{m}_1+1}|)^N})^{M_1}\quad\textrm{for}\quad n\in\ZZ.
$$
Since $|\lambda_{1,\tilde{m}_1}|>|\lambda_{1,\tilde{m}_1+1}|$
taking $N\to\infty$
we conclude that $\mathbf{\varphi_1}\in(\Oo(\hat{S}_{(d+\eta+2n\pi/p)/\kappa}))^{M_1}$ for $n\in\ZZ$.

\medskip\par
To prove the exponential growth, for given union of subsectors $S'\prec \bigcup_{n\in\ZZ}\hat{S}_{(d+\eta+2n\pi/p)/\kappa}$ 
and $S''\prec\bigcup_{n\in\ZZ}\hat{S}_{(d+2n\pi)\kappa}$, fix $A,B<\infty$ such that
$$
\|\mathbf{\Phi}(z)\|\leq Ae^{B|z|^K}\quad\textrm{for}\quad z\in S',
$$
$$
\|\mathbf{\varphi_{11}}(z)\|\leq Ae^{B|z|^K}
\quad\textrm{for}\quad z\in S'\cap D_R
$$
and
$$
\|\mathbf{\tilde{\Psi}}(t)\|\leq Ae^{B|\lambda_{1,\tilde{m}_1+1}|^K|t|^K}\quad\textrm{for}\quad t\in S'',
$$
where the norm $\|\cdot\|$ is defined as
$$
\|f(z)\|=\max\{|f_1(z)|,\dots,|f_n(z)|\}\quad\textrm{for}\quad f(z)=(f_1(z),\dots, f_n(z)).
$$
By the first equation in (\ref{eq:system_solved}) for every
$\varepsilon>0$
there exists
$C_1<\infty$ such that
$$
\|\mathbf{\varphi_{12}}(z)\|\leq C_1Ae^{B(1+\varepsilon)^K|z|^K}
\quad\textrm{for}\quad z\in S'\cap D_R.
$$
Moreover, by Lemma \ref{le:crucial} for given $\varepsilon>0$ there exists $C_2<\infty$
such that 
$$
\|\mathbf{\psi_{12}}(t)\|\leq C_2C_1Ae^{B(1+\varepsilon)^{2K}|\lambda_{1,\tilde{m}_1+1}|^K|t|^K}
\quad\textrm{for}\quad t\in S''\cap D_{R/|\lambda_{1,\tilde{m}_1+1}|}.
$$
Next, by the second equation in (\ref{eq:system_solved}) for given $\varepsilon>0$ there exists $C_3<\infty$
such that 
$$
\|\mathbf{\psi_{11}}(t)\|\leq C_3C_2C_1Ae^{B(1+\varepsilon)^{3K}|\lambda_{1,\tilde{m}_1+1}|^K|t|^K}
\quad\textrm{for}\quad t\in S''\cap D_{R/|\lambda_{1,\tilde{m}_1+1}|}.
$$
Applying once again Lemma \ref{le:crucial} we see that
for given $\varepsilon>0$ there exists $C_4<\infty$ satisfying
$$
\|\mathbf{\varphi_{11}}(z)\|\leq C_4C_3C_2C_1Ae^{B(1+\varepsilon)^{4K}\big|\frac{\lambda_{1,\tilde{m}_1+1}}{\lambda_{1,\tilde{m}_1}}\big|^K |z|^K}
\quad\textrm{for}\quad z\in S'\cap D_{R\big|\frac{\lambda_{1,\tilde{m}_1+1}}{\lambda_{1,\tilde{m}_1}}\big|}.
$$
Since $|\lambda_{1,\tilde{m}_1}|>|\lambda_{1,\tilde{m}_1+1}|$ we may choose such sufficiently small $\varepsilon>0$ that
$$
\frac{(1+\varepsilon)^4|\lambda_{1,\tilde{m}_1+1}|}{|\lambda_{1,\tilde{m}_1}|} \leq 1.
$$
Putting $\tilde{C}:=C_4C_3C_2C_1$ for such $\varepsilon>0$ and repeating our considerations $N$ times we get the estimation
\begin{equation}
\label{eq:estim}
\|\mathbf{\varphi_{11}}(z)\|\leq \tilde{C}^NAe^{B|z|^K}
\quad\textrm{for}\quad z\in S'\cap D_{R\big|\frac{\lambda_{1,\tilde{m}_1+1}}{\lambda_{1,\tilde{m}_1}}\big|^N}.
\end{equation}
Hence, if we take $p\in\NN$ such that
$\tilde{C}\leq\Big|\frac{\lambda_{1,\tilde{m}_1+1}}{\lambda_{1,\tilde{m}_1}}\Big|^p$,
we deduce by (\ref{eq:estim}) that there exists $\tilde{A}<\infty$ such that
$$\|\mathbf{\varphi_{11}}(z)\|\leq \tilde{A}(|z|^p+1)e^{B|z|^K}\quad\textrm{for}\quad z\in S'.$$
So, by (\ref{eq:system_solved}) we see that $\mathbf{\varphi}(z)=(\mathbf{\varphi_1}(z),\mathbf{\varphi_2}(z))\in\Big(\Oo^{K}\big(\bigcup_{n\in\ZZ}\hat{S}_{(d+\eta+2n\pi/p)/\kappa}\big)\Big)^M$.

Finally, by \cite[Theorem 2]{Mic8} we conclude that $w(t,z)\in\Oo^{K}(\hat{S}_{(d+2l\pi)/\kappa}\times\hat{S}_{(d+\eta+2n\pi/p)/\kappa})$
for $l,n\in\ZZ$.
\end{proof}
\bigskip\par
\section{Analytic continuation of solutions}
We apply Theorem \ref{th:2} to prove the result about analytic continuation of the solution of the Goursat problem.
To this end we assume that the Newton polygon of
the operator $P(\partial_{t},\partial_{z})$ is equal to $Q(M,-M)$,
i.e. has exactly one horizontal and one vertical side  with a  vertex at the point $(M,-M)$, where $M$ is an order of the operator $P(\partial_t,\partial_z)$ with respect to $\partial_t$. In other words this operator has characterisation
$$
P(\partial_t,\partial_z)=P_{(1)}(\partial_t,\partial_z)
P_{(2)}(\partial_t,\partial_z),$$
where 
$$
P_{(1)}(\partial_t,\partial_z)=\prod_{k=1}^{M_1}(\partial_t-\lambda_{k}(\partial_z))\quad\textrm{with}\quad \lambda_{k}(\zeta)\sim \lambda_{k}\zeta
$$
for some $M_1\in\NN$, and
$$
P_{(2)}(\partial_t,\partial_z)=\prod_{k=1}^{M_2}(\partial_t-\tilde{\lambda}_{k}(\partial_z))\quad\textrm{with}\quad \tilde{\lambda}_{k}(\zeta)\sim \tilde{\lambda}_{k}\zeta^{q_k},\quad q_k\leq q
$$
for some $M_2\in\NN_0$ and $q<1$.
Here $M=M_1+M_2$.

Additionally, without loss of generality we may assume that 
$|\lambda_1|\geq|\lambda_2|\geq...\geq|\lambda_{M_1}|.$

\begin{Th}
 \label{th:3}
  We assume that $K\geq \frac{1}{1-q}$, $(j,\alpha)\in{\sf conv\,}\{\mathring{N}_0\}$, $|\lambda_{m-j}|>|\lambda_{m-j+1}|$  and $N$-th finite section Toeplitz matrix $T_{f}(N)$ with symbol $f(z):=z^jf_0(z)$ is invertible for any $N\in\NN_0$.
  We also assume $d,\eta\in\RR$, $p\in\NN$ and  $\arg\lambda_{l}\in\{\eta+\frac{2n\pi}{p}\colon n\in\ZZ\}$ for $l=1,\dots,m_1$.
  \par
Then there exists the unique solution $w(t,z)\in\Oo(D^2)$ of the Goursat problem
\begin{equation}
  \label{eq:goursat_problem_0}
  \left\{
   \begin{array}{l}
    P(\partial_{t},\partial_{z})w(t,z)=f(t,z)\in\Oo(D^2)\\
     w(t,z)-v(t,z)=O(t^j z^{\alpha}),\ \ v(t,z)\in\Oo(D^2),
   \end{array}
  \right.
\end{equation}
which belongs to the space
$\Oo^{K,K}(\hat{S}_d\times\hat{S}_{d+\eta+2n\pi/p})$ for $n\in\ZZ$
if and only if the Goursat data $v(t,z)$ and the inhomogeneity $f(t,z)$ belong also to the same space.
 \end{Th}
\begin{proof}
First, observe that by Theorem \ref{th:2} there exists the unique solution $w(t,z)\in\Oo(D^2)$ of the Goursat problem (\ref{eq:goursat_problem_0}).
\par
We will show the equivalence between the analytic continuation property of the solution $w(t,z)$ and the same property of the Goursat data $v(t,z)$ and the inhomogeneity $f(t,z)$.

 ($\Longrightarrow$)
 Assume that the solution $w(t,z)$ belongs to the space $\Oo^{K,K}(\hat{S}_d\times\hat{S}_{d+\eta+2n\pi/p})$ for $n\in\ZZ$.
 Since this space is closed under derivations $\partial_t$ and $\partial_z$,
 we conclude that also $f(t,z)=P(\partial_{t},\partial_{z})w(t,z)$
belongs to the same space. Moreover
$\varphi_i(z):=\partial_t^i w(0,z)\in\Oo^{K}(\hat{S}_{d+\eta+2n\pi/p})$ for $n\in\ZZ$, $i=0,...,j-1$, and
$\psi_{\beta}(t):=\partial_z^\beta w(t,0)\in\Oo^{K}(\hat{S}_d)$ for $\beta=0,...,\alpha-1$. Hence also $v(t,z)$, given by (\ref{eq:any_v}) with $r(t,z)=0$, belongs to the same space.
\bigskip\par
($\Longleftarrow$)
\textit{Step 1. Reduction to the homogeneous equation.}
First, we consider the Cauchy problem
\begin{equation*}
  \left\{
   \begin{array}{l}
    P(\partial_{t},\partial_{z})\overline{w}(t,z)=f(t,z)\in\Oo(D^2)\\
     \partial_t^l\overline{w}(0,z)=0\quad\textrm{for}\quad l=0,...,M-1.
   \end{array}
  \right.
\end{equation*}
Since $f(t,z)\in\Oo^{K,K}(\hat{S}_d\times\hat{S}_{d+\eta+2n\pi/p})$ for $n\in\ZZ$ and $K\geq\frac{1}{1-q_k}$ for $k=1,...,M_2$, applying Lemma \ref{le:inhomogeneous} with $a=1$ and $b=q$ we conclude that $\overline{w}(t,z)\in\Oo^{K,K}(\hat{S}_d\times\hat{S}_{d+\eta+2n\pi/p})$ for $n\in\ZZ$.

Let $\tilde{w}(t,z):=w(t,z)-\overline{w}(t,z)$. By the linearity of (\ref{eq:goursat_problem_0}), we get that $\tilde{w}(t,z)$ is a solution of the homogeneous Goursat problem
\begin{equation*}
  \left\{
   \begin{array}{l}
    P(\partial_{t},\partial_{z})\tilde{w}(t,z)=0\\
     \tilde{w}(t,z)-\tilde{v}(t,z)=O(t^j z^{\alpha}),
   \end{array}
  \right.
\end{equation*}
where
$$\tilde{v}(t,z):=v(t,z)-\sum_{\beta=0}^{\alpha-1}\partial_z^{\beta}\overline{w}(t,0)\frac{z^{\beta}}{\beta!}$$
belongs to the space $\Oo^{K,K}(\hat{S}_d\times\hat{S}_{d+\eta+2n\pi/p})$, $n\in\ZZ$.

It means that it is enough to prove our assertion in the case when $f(t,z)\equiv 0$.
\medskip\par
\textit{Step 2. Application of Lemma \ref{le:crucial_2}.}
Let $\kappa\in\NN$ and $r_0\geq 0$ be such that
all algebraic functions $\lambda_{k}(\zeta)$ and $\tilde{\lambda}_k(\zeta)$ are
holomorphic in the variable $\zeta^{1/\kappa}$ for $|\zeta|\geq r_0$.

Applying Lemma \ref{le:crucial_2} with $a=1$ and $b=q$ we conclude that $\breve{w}(t,z):=w(t^{\kappa},z^{\kappa})$
satisfies the Goursat problem
\begin{equation*}
  \left\{
   \begin{array}{l}
    Q(\partial_{\Gamma_{1/\kappa},t},\partial_{\Gamma_{1/\kappa},z})\breve{w}(t,z)=0\\
     \breve{w}(t,z)-v(t^{\kappa},z^{\kappa})=O(t^{\kappa j} z^{\kappa\alpha}),
   \end{array}
  \right.
\end{equation*}
with $Q(\partial_{\Gamma_{1/\kappa},t},
\partial_{\Gamma_{1/\kappa},z}):=P(\partial^{\kappa}_{\Gamma_{1/\kappa},t},
\partial^{\kappa}_{\Gamma_{1/\kappa},z})$.
\medskip\par
\textit{Step 3. Application of Lemma \ref{le:crucial_3}.}
Observe that function $(t,z)\mapsto v(t^{\kappa},z^{\kappa})$ belongs to the space
$\Oo^{\kappa K,\kappa K}(\hat{S}_{(d+2l\pi)/\kappa}\times\hat{S}_{(d+\eta+2n\pi/p)/\kappa})$ for $l,n\in\ZZ$. Hence, applying Lemma \ref{le:crucial_3} with $a=1$ and $b=q$ we conclude that $\breve{w}(t,z)$ is contained at the same space. It means that $w(t,z)\in\Oo^{K,K}(\hat{S}_d\times\hat{S}_{d+\eta+2n\pi/p})$ for $n\in\ZZ$, so we get the assertion.
\end{proof}
\bigskip\par
 
The main idea and the difficult part of the above proof is hidden in Lemma \ref{le:crucial_3}. To show this idea in an easier way we calculate directly the simple example, which also holds by Theorem \ref{th:3}.
\begin{Ex}
Consider the Goursat problem
 \begin{equation}
  \label{eq:goursat_example}
  \left\{
   \begin{array}{l}
    (\partial_t-\lambda_1\partial_z)(\partial_t-\lambda_2\partial_z)w(t,z)=0\\
    w(0,z)=\varphi(z)\in\Oo(D)\\
    w(t,0)=\psi(t)\in\Oo(D).
   \end{array}
  \right.
 \end{equation}
Then the following conditions hold.
\begin{itemize}
 \item If $|\lambda_1|\neq|\lambda_2|$ then there exists a unique solution $w(t,z)\in\Oo(D^2)$ of (\ref{eq:goursat_example}). Moreover, this solution is given by the formula
 \begin{equation}
 \label{eq:formula_fg}
 w(t,z)=f(\lambda_1t+z)+g(\lambda_2 t+z)\quad\textrm{for some}\quad f(z),g(z)\in\Oo(D).
 \end{equation}
 \item Assume additionally that $\arg\lambda_1=\arg\lambda_2=\eta$ and $K>0$. Then the following conditions are equivalent
 \begin{enumerate}
  \item $\varphi(z)\in\Oo^{K}(\hat{S}_{d+\eta})$ and $\psi(t)\in\Oo^K(\hat{S}_d)$,
  \item $f(z),g(z)\in\Oo^K(\hat{S}_{d+\eta})$,
  \item $w(t,z)\in\Oo^K(\hat{S}_d\times D)$.
 \end{enumerate}
\end{itemize}
\end{Ex}
\begin{proof}
The first statement is given directly by Theorem \ref{th:2} and by the observation that the general solution $u(t,z)\in\Oo(D^2)$ of the equation
$(\partial_t-\lambda\partial_z)u(t,z)=0$ is given by $u(t,z)=F(\lambda t +z)$ for some $F\in\Oo(D)$.

To prove the second statement,
observe that the implications (2) $\Rightarrow$ (3) and (2) $\Rightarrow$ (1) are given immediately by the formula (\ref{eq:formula_fg}).
To show the implication (3) $\Rightarrow$ (2), note that by (\ref{eq:formula_fg}) we get 
\begin{equation*}
  \left\{
   \begin{array}{l}
    \partial_tw(t,0)=\lambda_1f'(\lambda_1 t)+\lambda_2g'(\lambda_2t)=:F_1(t)\in\Oo^{K}(\hat{S}_{d})\\
    \partial_zw(t,0)=f'(\lambda_1t)+g'(\lambda_2 t)=:F_2(t)\in\Oo^{K}(\hat{S}_{d}).
   \end{array}
   \right.
 \end{equation*}
Hence $f'(\lambda_1t)=\frac{1}{\lambda_1-\lambda_2}(F_1(t)-\lambda_2 F_2(t))\in\Oo^K(\hat{S}_{d})$ and also $g'(\lambda_2t)=\frac{1}{\lambda_2-\lambda_1}(F_1(t)-\lambda_1 F_2(t))\in\Oo^K(\hat{S}_{d})$. It means that $f(z),g(z)\in\Oo^K(\hat{S}_{d+\eta})$.

So, it is sufficient to show the implication (1) $\Rightarrow$ (2).
  Since $w(t,z)$ given by (\ref{eq:formula_fg}) satisfies the conditions from (\ref{eq:goursat_example}) we get
   \begin{equation*}
  \left\{
   \begin{array}{l}
    w(0,z)=f(z)+g(z)=\varphi(z)\in\Oo^{K}(\hat{S}_{d+\eta})\\
    w(t,0)=f(\lambda_1t)+g(\lambda_2 t)=\psi(t)\in\Oo^{K}(\hat{S}_{d}).
   \end{array}
  \right.
 \end{equation*}
 Without loss of generality we may assume that $|\lambda_1|<|\lambda_2|$. Then for $\tau:=\lambda_1/\lambda_2\in(0,1)$ and $z=\lambda_2 t$ we get
 $$
 f(z)-f(\tau z)=\varphi(z)-\psi(z/\lambda_2)=:F(z)\in\Oo^K(\hat{S}_{d+\eta}).
 $$
 Applying $N$-times this equality we conclude that
 $$
 f(z)=\sum_{n=0}^{N-1}F(\tau^n z)+f(\tau^N z).
 $$

 We know that there exists $R>0$ such that $f(z)\in\Oo(D_R)$. For fixed $z\in\CC$ we take $N\in\NN$ such that $\tau^Nz\in D_R$ and
 $\tau^N|z|<R\leq \tau^{N-1}|z|$. Hence $f(z)\in\Oo(\hat{S}_{d+\eta})$.
 \bigskip\par
  To estimate $|f(z)|$ observe that we may take $A,B,C<\infty$ such that $|F(z)|\leq Ae^{B|z|^K}$ and $|f(z)|\leq C$ for $z\in D_R$. Additionally we may assume that
  $B>1/R^K$.
 \bigskip\par
 We have
 \begin{equation*}
  |f(z)|\leq\sum_{n=0}^{N-1}|F(\tau^nz)|+|f(\tau^N z)|\leq
  \sum_{n=0}^{N-1}Ae^{B\tau^{nK}|z|^K}+C.
 \end{equation*} 

 Since $BR^K>1$, we may use the inequality
  \begin{equation*}
 e^{a_1}+\dots+e^{a_n}\leq e^{a_1+\dots+a_n}\quad\textsf{for}\quad
 a_1,\dots,a_n\geq 1
 \end{equation*}
 to conclude that
 \begin{equation*}
  |f(z)|
  \leq Ae^{\sum_{n=0}^{N-1}B\tau^{nK}|z|^K}+C
  \leq (C+A)e^{\frac{B}{1-\tau^K}|z|^K}.
  \end{equation*}
 
It means that $f(z)\in\Oo^K(\hat{S}_{d+\eta})$.
Analogously we show
that also $g(z)\in\Oo^K(\hat{S}_{d+\eta})$. 
\end{proof}
\bigskip\par
\section{Summable solutions}
In this section we prove the main result of the paper, where
we find the conditions for summable solutions of the Goursat problem in terms of
the analytic continuation property of the Borel transform
of the inhomogeneity and the Goursat data.

Here we assume that the Newton polygon $N(P)$ of
the operator $P(\partial_{t},\partial_{z})$ has exactly one side with a positive slope, which is equal to $K>0$, $K=\frac{1}{s}$, and that the point $(M,-M)$ is one of the vertex of $N(P)$, where $M$ is an order of the operator $P(\partial_t,\partial_z)$ with respect to $\partial_t$.
It means that we may write the operator $P$ as
$$
P(\partial_t,\partial_z)=P_{(1)}(\partial_t,\partial_z)
P_{(2)}(\partial_t,\partial_z),$$
where 
$$
P_{(1)}(\partial_t,\partial_z)=\prod_{k=1}^{M_1}(\partial_t-\lambda_{k}(\partial_z))\quad\textrm{with}\quad \lambda_{k}(\zeta)\sim \lambda_{k}\zeta^{q},\quad q=1+\frac{1}{K}>1,
$$
and
$$
P_{(2)}(\partial_t,\partial_z)=\prod_{k=1}^{M_2}(\partial_t-\tilde{\lambda}_{k}(\partial_z))^{\beta_k}\quad\textrm{with}\quad \tilde{\lambda}_{k}(\zeta)\sim \tilde{\lambda}_{k}\zeta^{q_k},\quad q_k\leq 1
$$
with $M_1+M_2=M$.

Additionally, without loss of generality we may assume that 
$|\lambda_1|\geq|\lambda_2|\geq...\geq|\lambda_{M_1}|.$
\par\bigskip
Now, we are ready state the main result of the paper
  \begin{Th}[Main theorem]
  \label{th:4}
  We assume that $(j,\alpha)\in{\sf conv\,}\{\mathring{N}_s\}$,
  $|\lambda_{M_1-j}|>|\lambda_{M_1-j+1}|$ and $N$-th finite section Toeplitz matrix $T_{f}(N)$ with symbol $f(z):=z^jf_s(z)$ is invertible for any $N\in\NN_0$.
  We also assume that $\eta\in\RR$, $p\in\NN$ and 
$\arg\lambda_{l}\in \{\eta+\frac{2n\pi}{p}\colon n\in\ZZ\}$ for $l=1,\dots,M_1$.
\par
Then there exists a unique solution $u(t,z)\in\Oo[[t]]_s$ of the Goursat problem
\begin{equation}
  \label{eq:goursat_problem}
  \left\{
   \begin{array}{l}
    P(\partial_{t},\partial_{z})u(t,z)=f(t,z)\in\Oo[[t]]_s\\
     u(t,z)-v(t,z)=O(t^j z^{\alpha}),\ \ v(t,z)\in\Oo[[t]]_s.
   \end{array}
  \right.
\end{equation}

Moreover $\Bo_s u(t,z)\in\Oo^{K,qK}(\hat{S}_d\times\hat{S}_{(d+\eta+2n\pi/p)/q})$  for $n\in\ZZ$ if and only if $\Bo_s f(t,z)$ and $\Bo_s v(t,z)$ belong to the same space.

In particular, if the inhomogeneity $f(t,z)$ and 
the Goursat data $v(t,z)$ satisfy condition
$$\Bo_s f(t,z),\Bo_s v(t,z)\in\Oo^{K,qK}(\hat{S}_d\times\hat{S}_{(d+\eta+2n\pi/p)/q})\quad\textrm{for}\quad n\in\ZZ$$
then $u(t,z)$ is $k$-summable in the direction $d$
 \end{Th}
 \begin{proof}
We divide the proof of the main theorem into the following steps.
\bigskip\par
\textit{Step 1.} By Theorem \ref{th:2} the Goursat problem (\ref{eq:goursat_problem}) has the unique solution $u(t,z)\in\Oo[[t]]_s$.
\bigskip\par
\textit{Step 2.} Assume that the solution $\Bo_su(t,z)$ belongs to the space $\Oo^{K,qK}(\hat{S}_d\times\hat{S}_{(d+\eta+2n\pi/p)/q})$ for $n\in\ZZ$.
 Since this space is closed under derivations $\partial_{\Gamma_q,t}$ and $\partial_z$,
 we conclude that also 
 $$\Bo_sf(t,z)=\Bo_sP(\partial_{t},\partial_{z})u(t,z)=P(\partial_{\Gamma_q,t},\partial_z)\Bo_su(t,z)
 $$
belongs to the same space. Moreover
$\varphi_i(z):=\partial_{\Gamma_q,t}^i w(0,z)\in\Oo^{qK}(\hat{S}_{(d+\eta+2n\pi/p)/q})$ for $n\in\ZZ$, $i=0,...,j-1$, and
$\psi_{\beta}(t):=\partial_z^\beta \Bo_s w(t,0)\in\Oo^{K}(\hat{S}_d)$ for $\beta=0,...,\alpha-1$. Hence also $\Bo_sv(t,z)$ belongs to the same space, where $v(t,z)$ is given by (\ref{eq:any_v}) with $r(t,z)=0$.
\bigskip\par
\textit{Step 3.}
By the linearity, if $u(t,z)$ is a solution of (\ref{eq:goursat_problem}) then $u(t,z)=u_1(t,z)+u_2(t,z)$,
 where $u_1(t,z)$ and $u_2(t,z)$ satisfy
 \begin{equation*}
  \left\{
   \begin{array}{l}
    P(\partial_{t},\partial_{z})u_1(t,z)=f(t,z)\\
     \partial_t^lu_1(0,z)=0\quad\textrm{for}\quad l=0,...,M-1
   \end{array}
  \right.
\end{equation*}
and
\begin{equation*}
  \left\{
   \begin{array}{l}
    P(\partial_{t},\partial_{z})u_2(t,z)=0\\
     u_2(t,z)-v_2(t,z)=o(t^j z^{\alpha})
   \end{array}
  \right.
\end{equation*}
with 
\begin{equation*}
v_2(t,z):=v(t,z)-\sum_{l=0}^{\alpha-1}\frac{z^l}{l!}\partial_z^lu_1(t,0).
\end{equation*}
To finish the proof it is enough to show that
$\Bo_s u_1(t,z),\Bo_s u_2(t,z)\in\Oo^{K,qK}(\hat{S}_d\times\hat{S}_{(d+\eta+2n\pi/p)/q})$ for $n\in\ZZ$.
\bigskip\par
\textit{Step 4.} We will show that
if $\Bo_{s}f(t,z)\in\Oo^{K,qK}(\hat{S}_d\times\hat{S}_{(d+\eta+2n\pi/p)/q})$ for $n\in\ZZ$ then $\Bo_{s}u_1(t,z)$ belongs to the same space.

To this end we denote $w_1(t,z):=\Bo_{s}u_1(t,z)$ and $g(t,z):=\Bo_{s}f(t,z)\in\Oo^{K,qK}(\hat{S}_d\times\hat{S}_{(d+\eta+2n\pi/p)/q})$, $n\in\ZZ$. Observe that $w_1(t,z)$ is a solution of the equation
 \begin{equation*}
  \left\{
   \begin{array}{l}
     P(\partial_{\Gamma_q,t},\partial_z)w_1(t,z)=g(t,z)\\
     \partial_{\Gamma_{q},t}^lw_1(0,z)=0\ l=0,\dots,M-1.
   \end{array}
  \right.
  \end{equation*}
Applying Lemma \ref{le:inhomogeneous} with $a=q$ and $b=1$
we conclude that $\Bo_s u_1(t,z)=w_1(t,z)\in\Oo^{K,qK}(\hat{S}_d\times\hat{S}_{(d+\eta+2n\pi/p)/q})$, $n\in\ZZ$.
\bigskip\par
\textit{Step 5.}
By Lemma \ref{le:crucial_2} with $a=q$ and $b=1$ the function $w_2(t,z):=(\Bo_{s} u_2)(t^{q\kappa},z^{\kappa})$
 satisfies the Goursat problem
 \begin{equation}
  \label{eq:w}
  \left\{
   \begin{array}{l}
    P(\partial^{q\kappa}_{\Gamma_{1/\kappa},t},\partial^{\kappa}_{\Gamma_{1/\kappa},z})w_2(t,z)=0\\
     w_2(t,z)-\tilde{v}_2(t,z)=O(t^{q\kappa j} z^{\kappa\alpha}),
   \end{array}
  \right.
\end{equation}
where $\tilde{v}_2(t,z)=(\Bo_{s} v_2)(t^{q\kappa},z^{\kappa})\in\Oo^{qK\kappa,qK\kappa}(\hat{S}_{(d+2l\pi)/q\kappa}\times\hat{S}_{(d+\eta+2n\pi/p)/q\kappa})$ for $l,n\in\ZZ$.
\bigskip\par
\textit{Step 6.}
By Lemma \ref{le:crucial_3} 
 we conclude that the solution $w_2(t,z)$ of the Goursat problem (\ref{eq:w}) 
 belongs to the space $\Oo^{q\kappa K,q\kappa K}(S_{(d+2l\pi)/q\kappa}\times S_{(d+\eta+2n\pi/p)/q\kappa})$ for $l,n\in\ZZ$.
 Since $w_2(t,z)=(\Bo_{s} u_2)(t^{q\kappa},z^{\kappa})$ we deduce that
 $\Bo_{s} u_2(t,z)$ belongs to the space $\Oo^{K,qK}(\hat{S}_{d}\times \hat{S}_{(d+\eta+2n\pi/p)/q})$ for $n\in\ZZ$.
 \end{proof}
 \bigskip\par
 We will show that under some additional conditions on the operator $P(\partial_t,\partial_z)$, for the homogeneous Goursat problem the necessary condition for the summability given in Theorem \ref{th:4} is also sufficient. It gives us in this case the characterisation of summable solutions of the Goursat problem in terms of the Goursat data.
 \begin{Th}
 \label{th:5}
Under the assumptions of the above theorem, if additionally $P(\partial_t,\partial_z):=P_{(1)}(\partial_t,\partial_z)$, $f(t,z)\equiv 0$ and $p=1$
then $u$ is $K$-summable in a direction $d$
if and only if $\Bo_{s}v(t,z)\in\Oo^{K,qK}(\hat{S}_d\times\hat{S}_{(d+\eta+2n\pi)/q})$ for $n\in\ZZ$.
\end{Th}
\begin{proof}
 ($\Longrightarrow$) 
 We may assume that the operator $P(\partial_t,\partial_z)$ is factorised as
$$
P(\partial_t,\partial_z)=\prod_{k=1}^{m_1}(\partial_t-\lambda_{1k}(\partial_z))^{\alpha_{1k}},
$$
for some algebraic functions $\lambda_{1k}(\zeta)$ which are analytic functions of the variable 
 $\xi=\zeta^{1/\kappa}$ for $|\zeta|\geq r_0$ satisfying 
 $\lambda_{1k}(\zeta)\sim\lambda_{1k}\zeta^{q}$, $q=1+s>1$ and $\arg\lambda_{1k}=\eta$ for $k=1,...,m_1$.
 
 Since $u$ is $K$-summable in a direction $d$ we see that $w(t,z):=\Bo_su(t,z)\in\Oo^K(\hat{S}_d\times D)$. By \cite[Theorem 1]{Mic8} $w(t,z)=\sum_{k=1}^{m_1}\sum_{l=1}^{\alpha_{1k}}w_{kl}(t,z)$, where $w_{kl}(t,z)$ satisfies the equation
  \begin{equation*}
  \left\{
   \begin{array}{l}
    (\partial_{\Gamma_q,t}-\lambda_{1k}(\partial_{z}))^l w_{kl}(t,z)=0\\
     \partial_{\Gamma_q,t}^jw_{kl}(0,z)=0\quad\textrm{for}\quad j=0,...,l-2\\
    \partial_{\Gamma_q,t}^{l-1}w_{kl}(0,z)=\varphi_{kl}(z)\in\Oo_{1/\kappa}(D).     
   \end{array}
  \right.
\end{equation*}
By \cite[Theorem 1]{Mic8} also $w_{kl}(t,z)\in\Oo^K_{1,1/\kappa}(\hat{S}_d\times D)$.
It means also by \cite[Theorem 2]{Mic8} that $w_{kl}(t,z)\in\Oo^{K,qK}_{1,1/\kappa}(\hat{S}_d\times \hat{S}_{(d+\eta+2n\pi)/q})$ for $n\in\ZZ$.
Hence, finally we conclude that $w(t,z)\in\Oo^{K,qK}(\hat{S}_d\times \hat{S}_{(d+\eta+2n\pi)/q})$ for $n\in\ZZ$.
\bigskip\par
($\Longleftarrow$) The implication is given immediately by Theorem \ref{th:4}.
 
\end{proof}
\begin{Rem}
 Observe that in the special case when the Goursat problem reduces to the Cauchy problem, the above theorem  is a generalisation of the result of Lutz, Miyake and Sch\"afke \cite{L-M-S} for the heat equation, the result of Miyake \cite{Miy} for the equation $\partial_t^pu-\partial_z^qu=0$ with $p<q$,
 and the result of Ichinobe \cite{I} for the quasi-homogeneous equations. 
\end{Rem}
\bigskip\par
\section{Final remarks}
In this paper we make the first step in the study of the summable solutions of the Goursat problem. Our next aim will be 
to generalise the main theorem to the case when the Newton polygon of the operator $P(\partial_t,\partial_z)$ in the Goursat problem has several positive slopes.

In the future we also plan to find the similar conditions for multisummable solutions of the Goursat problem for the general linear partial differential equations with constant coefficients.
\bigskip\par
\bibliographystyle{siam}
\bibliography{summa}
\end{document}